\documentclass[11pt,letter,reqno]{amsart}
\usepackage[latin9]{inputenc}
\usepackage{amstext}
\usepackage{amsthm}
\usepackage{amssymb}
\usepackage{graphicx}
\usepackage{comment} 
\makeatletter
\numberwithin{equation}{section}
\numberwithin{figure}{section}
\theoremstyle{plain}

 \theoremstyle{remark}

\newtheorem{theorem}{Theorem}[section]

\newtheorem{lemma}[theorem]{Lemma}

\newtheorem{example}[theorem]{Example}

\newtheorem{assumption}[theorem]{Assumption}
\newtheorem{remark}[theorem]{Remark}

\usepackage{amsthm}
\usepackage{color}
\usepackage{color}\usepackage{xcolor}
\usepackage{subcaption}
\usepackage{enumerate}
\usepackage[draft]{todonotes}
\usepackage[american]{babel}
\usepackage{geometry}

\newcommand{\ri}{\mathrm{i}}
\newcommand{\eig}{\mathrm{eig}}
\newcommand{\spec}{\mathrm{spec}}

\newcommand{\diag}{\mathrm{diag}}

\def\A{{\mathcal A}}
\def\C{{\mathcal C}}

\def\O{{\mathcal O}}
\def\U{{\mathcal U}}
\def\eps{\varepsilon}
\def\th{\theta}
\def\hK{{\hat K}}

\def\nat{{\mathbb N}}
\def\real{{\mathbb R}}
\def\complex{{\mathbb C}}
\def\integer{{\mathbb Z}}
\def\circle{{\mathbb S}}

\def\Lip {{ \rm Lip}}
\def\Id {{ \rm Id}}

\makeatother

  \providecommand{\remarkname}{Remark}
\providecommand{\theoremname}{Theorem}

\begin{document}

\title[Perturbation of Lyapunov-Subcenter-Manifolds]{Global Persistence of Lyapunov-Subcenter-Manifolds as Spectral Submanifolds under Dissipative Perturbations}

\begin{abstract}

For a nondegenerate analytic system with a conserved quantity, a classic result by Lyapunov guarantees the existence of an analytic manifold of periodic orbits tangent to any two-dimensional, elliptic eigenspace of a fixed point satisfying nonresonance conditions. These two dimensional manifolds are referred to as Lyapunov Subcenter
Manifolds (LSM).\\
Numerical and experimental observations in the nonlinear vibrations literature suggest that LSM's often persist under autonomous, dissipative perturbations. These perturbed manifolds are useful since they provide information of the asymptotics of the convergence to equilibrium. \\
In this paper, we formulate and prove  precise mathematical results on the persistence of LSMs under dissipation. We show that, for Hamiltonian systems under mild non-degeneracy conditions on the perturbation, for small enough dissipation, there are analytic invariant manifolds of 
the perturbed system  that approximate (in the analytic sense) the LSM in a fixed neighborhood.  We provide examples 
that show that some non-degeneracy conditions on the perturbations are needed 
for the results to hold true. \\
We also study the dependence of the manifolds on the dissipation parameter. If $\eps$ is the dissipation parameter, we show that the manifolds  are real analytic in  $(-\eps_0, \eps_0) \setminus \{0\} $ and $C^\infty $  in  $(-\eps_0, \eps_0)$. We construct explicit asymptotic 
expansions in powers of $\eps$ (which presumably do not converge).\\
Finally, we present applications of our results  to a mechanical systems. 
\end{abstract}

\author[R.de la Llave]{Rafael de la Llave}
\thanks{R.L. Supported in part by NSF grant DMS-1800241} 
\address{School of Mathematics, Georgia Institute of Technology, Atlanta,
Georgia 30332-0160}
\email{rafael.delallave@math.gatech.edu}

\author[F. Kogelbauer]{Florian Kogelbauer}
\address{Institute for Mechanical Systems, ETH Z\"urich, Leonhardstrasse 21,
8092 Z\"{u}rich, Switzerland}
\email{floriank@ethz.ch}



\maketitle

\today


\section{Introduction}

As shown first by Lyapunov \cite[\S 42 p. 376]{liapounoff1907probleme}, given an analytic ODE with a non-degenerate conserved quantity (for example, a Hamiltonian system), near a fixed  point whose linearization contains a pair of complex conjugate imaginary eigenvalues which do not resonate with the other eigenvalues (see precise definition later), we can find a one dimensional family of periodic orbits (hence a two dimensional invariant manifold) tangent at the origin to the eigenspace corresponding to the the pair of imaginary
eigenvalues. \\
One can think of these families of periodic orbits as a nonlinear analogue of the periodic orbits predicted by the linearization, i.e., the harmonic oscillator. These two-dimensional manifolds are often referred to as\textit{ Lyapunov subcenter manifolds} (LSMs) because they are submanifolds of the full center manifold of the fixed point.\\
The proof of \cite{liapounoff1907probleme}, is based on constructing perturbative series starting at the origin and then showing that they converge. The notation of  \cite{liapounoff1907probleme} may  require some effort for modern readers and the formal statement is only for Hamiltonian systems.
See \cite{kelley1967liapounov, MoserZ, MeyerH} or later in this paper for modern proofs based on the elementary implicit function theorem.\\
More modern proofs for Hamiltonian systems based on estimating series can be found in \cite{Moulton20},  \cite[II \S16 p. 104]{SiegelM}
or on normal form theory \cite{kelley1969analytic}. 
The paper \cite{moser1958generalization} contains a generalization of the convergence of the normal forms in the saddle-center case. 
Results (based on variational methods) on  the existence of periodic orbits that do not require 
non-resonance assumptions but assume  positive definiteness on the conserved 
quantity appear in \cite{weinstein73, moser1976periodic}. The papers
\cite{Duistermaat1972, SchmidtS73} use averaging methods to prove generalizations of Lyapunov theorems for resonant systems.
A useful review of these results, along with further material, is in \cite{sijbrand1985properties}. Most of the results above on families of quasi-periodic solutions work for systems with finite 
differentiability.\\
None of the above  results, however, apply  under the addition of the slightestdissipative perturbation. Such small dissipative perturbations are present in applications to  mechanical vibrations. Perturbation arguments based on normal hyperbolicity  for the continuation of LSMs as invariant sets
are also inapplicable, given that LSMs are not normally hyperbolic. The addition of a small dissipation to a system is a very singular 
perturbation since even the smallest perturbation destroys 
all periodic orbits completely.\\
There is, however, a strong indication that remnants of LSMs continue to organize the dynamics of mechanical systems under the
addition of small damping and forcing to their conservative limits in a domain whose size is independent of the perturbation. Specifically, dissipative backbone curves (observed periodic response amplitudes near resonances plotted as a function of an external forcing
frequency) of such systems are virtually indistinguishable from their
conservative backbone curves (amplitudes of periodic orbits in LSMs plotted as a function of their frequency) for small enough damping.
The paper \cite{TOUZE2006958} illustrate this relationship numerically, but also shows that conservative and dissipative backbone
curves start deviating noticeably from each other for larger damping values.\\
An experimental technique, the \emph{force appropriation method} \cite{PEETERS2011486, PEETERS20111227} is directly based on the observation that a periodic orbit on an LSM survives unchanged when an external forcing is selected to cancel out damping exactly along the orbit. The paper \cite{PEETERS2011486, PEETERS20111227} demonstrate that close enough to the fixed point, such a forcing can be approximately constructed for small enough linear damping. This provides an intuitive explanation for the observed closeness of conservative and dissipative backbone curves under small linear damping and harmonic forcing, at least near the unforced equilibrium. \\
Another experimental technique, the \emph{resonance decay method}, cf. \cite{KERSCHEN2006505} and \cite{PEETERS2011486}, also applies periodic external forcing to the lightly damped conservative system, then tunes the forcing frequency to reach a locally maximal amplitude for the system. At this frequency, the forcing is subsequently turned off, and the instantaneous amplitude-frequency diagram of the resulting
decaying oscillations is taken to be as an approximation of the backbone curve of the conservative system. Again, this approach implicitly assumes that after the forcing is turned off, solutions evolve close
to an LSM of the conservative system. \\
Independent of these developments, the recent theory of \emph{spectral submanifolds} (SSMs) offers an extension of the LSM concept to dissipative systems, cf. \cite{Haller2016, Kogelbauer2018, Szalai20160759}. Inspired by the nonlinear normal mode concept of \cite{SHAW1994319} and based on the abstract invariant manifold results of \cite{de1997invariant, 
Cab2003pam,Cab2003,CABRE2005444}, the theory of SSMs guarantees the existence of a unique analytic, two-dimensional invariant manifold tangent to any two-dimensional, nonresonant eigenspace of a linearly asymptotically stable fixed point. Such eigenspaces arise from two-dimensional center subspaces of conservative systems under small dissipative perturbations. The question we address in this paper is whether indeed LSM continue into SSMs.\\
A mathematically trivial (and not very interesting 
physically) theory of persistence of LSMs is to observe that, if the eigenvalue pair corresponding to the LSM,  becomes dissipative, it has to remain 
non-resonant.  Then, we can obtain a SMM  
because of the theory of \cite{de1997invariant, Cab2003, CABRE2005444}. Since the coefficients of the expansion are obtained recursively by algebraic operations, we obtain easily that the Taylor coefficients of the SSM converge to that of the LSM as the dissipation goes to zero.\\
The physical shortcomings of the simple 
theory explained above come from the observation that, since the contraction along the manifold is becoming weaker as the dissipation becomes weaker, the theory of  \cite{de1997invariant, Cab2003,CABRE2005444} only guarantees the existence of a manifold whose size goes to zero as the dissipation vanishes. The physical usefulness of  
manifolds whose  size goes to zero with the dissipation is rather tenuous. See also Remark~\ref{rem:easy}.\\
Therefore, the question we address in this paper is the existence of SMMs whose size is independent of the value of the dissipation parameter and which converges to the LSM in a fixed domain in the sense of convergence of analytic manifolds. We present two results. Assuming that the system is Hamiltonian, the first results shows in great generality that there are asymptotic expansions in powers of the dissipation of invariant manifolds. In the second result,also  for Hamiltonian systems and under some further assumptions in the first order perturbation, we show that there are indeed invariant manifolds of the system which, when the dissipation goes to zero, approximate (in the sense of real analytic manifolds) the LSM in a neighborhood 
of fixed size independent of the dissipation.\\ We will also present examples that show that some version of the non-degeneracy assumptions we make are necessary. This seems to be in accordance to the physical experiments.\\
in the proof, we will see that -- as customary in singular perturbation theory -- one needs to make assumptions on the leading effects of 
the perturbation. In our case, the assumptions are rather mild, but we show examples  that show that without any assumptions, we could have that the domains of the analytic  SMM decreases to zero as the dissipation vanishes (similar results happen for manifolds of 
finite order  regularity). \\
In this paper we will concentrate specially in situations where the unperturbed manifold is not hyperbolic. The perturbed manifolds that arise in our construction will be slow manifolds  that violate the standard rate conditions of Normally Hyperbolic manifolds (NHIM) in \cite{Fenichel74, Mane78} and therefore, their persistence properties are not based on the theory of Normally Hyperbolic  Invariant Manifolds. They use essentially the  fact that the manifolds are attached to a fixed point. The mathematical theory of slow manifolds and, a fortiori, the theory of SSMs, is very subtle.  Before the mathematical theory was settled, several of the complications and puzzlements
\footnote{Many of the puzzles on the theory of the slow manifolds arise from the fact that  some uniqueness results are based on 
long term behavior and others are based  on regularity near the origin. The two conditions used to produce uniqueness lead to different 
manifolds. In particular, the slow  manifolds of \cite{Fraser98,MaasP92} are, generically different since the former is based in smooth expansions at the origin and the later in expansions on asymptotic behavior at infinity. Both of them are unique under the appropriate conditions, which are, however incompatible in many systems. See the discussion of this point in \cite{de1997invariant}.} the difficulties of the theory of slow manifolds  were mapped out lucidly in \cite{lorenz1986existence,lorenz1987nonexistence,lorenz1992slow}.\\
Since we will have to rely on the theory of 
subcenter manifolds, we also refer to the pioneering papers on slow manifolds  \cite{de1995irwin,  irwin1980new} 
and \cite{poschel1986invariant} for invariant submanifolds of center manifold under nonresonance conditions (which are not satisfied by Hamiltonian systems). 
We refer to \cite{CABRE2005444} for a more extensive review of the literature on slow invariant manifolds up to 15 years ago.\\
Other similar singular perturbation theories have 
been considered in the literature. Indeed, we will use a method similar to those used in \cite{calleja2017domains}, for small dissipations proportional to velocity  and \cite{calleja2012construction, calleja2017response}, for strong dissipation and forcing. The techniques used in small dissipation problems are closely related to the techniques used for parabolic manifolds, which can be considered heuristically as systems with an infinitesimal dissipation. See, for example, \cite{baldoma2004exponentially} or \cite{baldoma2007parameterization} which, as this paper, is based on the parametrization method. In those papers (as  in the present one) the technique is to first develop a formal perturbation theory that produces an approximate solution, then develop an a-posteriori theory that produces a true solution.  In our case, this technique is crucial to obtain results in neighborhoods of uniform size. In the theoretical results, the approximate solutions in the a-posteriori theorem are those produced by the asymptotic expansions, but one could take as well the solutions produced by a numerical method. 

\subsection{Organization of the paper} 
\label{sec:organization}
n section 2, we recall the classical results on the LSM, including their proves. We also present a special coordinate system for Hamiltonian systems, that will prove useful in the later calculations.\\
Sections 3 and 4 contain the new results of the paper. An informal statement of our main result, Theorem \ref{mainThm}, is included in Section~\ref{sec:statement}. The formal result is included as Theorem~\ref{mainprop}. The proof of Theorem~\ref{mainprop} is completed in two steps in Section~\ref{sec:proof}. \\
In a first step, taken up in Section~\ref{sec:approximate}, we obtain formal asymptotic expansions for the SSM. These expansions may be of practical interest since they give approximations of the manifolds and of the dynamics up to order $|\eps|^{N+1}$ ($\eps$ being a parameter
that measures the strength of the dissipation) domains of size $\O(1)$. These expansions give very detailed information on the convergence to equilibrium under weak dissipation.\\ 
The second step of the proof of Theorem~\ref{mainprop}, taken up Section~\ref{sec:fixedpoint}, shows that near prediction of the formal expansions there is a true SSM. This is obtained by reformulating the problem of invariance as a fixed-point problem for a functional acting an appropriate function space. Even if the contraction is weak, we can obtain a fixed point starting the iterative step from the approximate solution obtained in the first stage.\\
In Section~\ref{sec:examples}, we provide some examples that show that some of our assumptions on the survival of LSMs as SSMs cannot be completely omitted. We present examples where there is no analytic (or even differentiable of high order) convergence in domains of size $\O(1)$.\\
Finally, in Section~\ref{sec:applications}, we illustrate our results on a concrete mechanical example. 

\subsection{Notation}

\def\BB{{\tilde B}^{n}_\delta}

Let
\begin{equation}
B_{\delta}^{n}=\{x\in\real^{n}:\,|x|<\delta\},
\end{equation}
denote the $n$-dimensional ball of radius $\delta$ around the origin of $\real^n$. For a sufficiently small $\tau > 0$, we will denote \begin{equation}\label{deftau}
\BB = \{x \in \complex^{n}|  d(x, B_\delta^{n}) < \tau\},
\end{equation} 
where, here, $d$ is the standard Euclidean distance. In the following, we will not write explicitly the parameter $\tau$ to avoid cluttering the notation.\\
Since the proofs we present will be based on soft methods (the implicit function theorem and contraction mappings in function spaces), the arguments for real values carry over to complex values as well. Of course, we need to take care of making sure that the domains match.\footnote{Even if the complex values of variables of parameters may not have a direct physical interpretation, they are indispensable to discuss analyticity properties. Of course, the physically relevant real values are particular cases of the complex ones and the results stated for complex sets apply to the real sets of physical interest. We will assume that the functions, even though they are defined for complex values, give real results for real arguments.}\\
For any Lipschitz continuous function $f:U\to\real^n$, defined on some open subset $U\subseteq\real^m$, we will denote its Lipschitz constant on $U$ as $\Lip(f)$.\\
For a matrix $A$, denote the full spectrum of $A$ by $\sigma(A)$. For a collection of eigenvalues $\lambda_{1},...,\lambda_{n}$ of an $n\times n$ matrix $A$, we denote the (generalized) eigenspace associated with $\lambda_{1},..,\lambda_{k}$ as 
\begin{equation}
\eig(\lambda_{1},...,\lambda_{j}):= 
{\rm Span} \{v\in\complex^{n}:(A-\lambda_kI)^lv=0\text{ for }1\leq k\leq n\text{ and some } l\geq 1\},\label{specss}
\end{equation}
which we will call a \textit{spectral subspace}. We write $\|A\|$ for the operator norm of $A$.\\
As it is well known, when $A$ is a real matrix and the sets of eigenvalues contains  the complex conjugates of all the its members, i.e. $\lambda_j^* \in \{ \lambda_m\}_{m = 1}^n$, then 
$\eig(\lambda_1, \cdots, \lambda_k)$ restricts to a real subspace of $\real^n$.\\
For two functions $f,g:\real^{n}\to\real$, we write $f(x)=\mathcal{O}(g(x))$ if there exists a constant $C>0$ such that $|f(x)|\leq C|g(x)|$ in a neighborhood of $x = 0$.\\
Let
\begin{equation}\label{Cdomain} 
\C_{\th}=\{\eps\in\complex |\,  \Re(\eps)  > 0|, \,  |\Im(\eps)|<\th |\Re(\eps)|\},
\end{equation}
for some $\th>0$, be a cone of complex numbers with width $\th$\footnote{The domain $\C_{\th}$ will play an important role in some of our results. The formal expansions in the dissipation parameter will be valid in domains of the form $\C_{\th}$. Note that the domains $\C_{\th}$ do not contain any ball centered at the origin, so that we do not show	that the expansions converge.}.

\subsection{Spaces of functions} 
Our upcoming contraction mapping arguments will require a careful definition of function spaces and norms. Specifically, we define for analytic functions $K:\BB\to\complex^n$ with $D^jK(0)=0, \text{ for }j=0,1,...,d-1$,
\begin{equation}\label{analyticnorm} 
\| K \|_{\A_{\delta,d} } = \sup_{z \in \BB \setminus\{0\} }
|z|^{-d  } | K(z)|,
\end{equation}
Let
\begin{equation}\label{defA}
\mathcal{A}_{\delta,d}
:=
\{K:\BB\to\complex^n: K\text{ is analytic }, D^jK(0)=0, \text{ for } j=0,1,...,d-1, \| K\|_{\mathcal{A}_{\delta, d}} <\infty \},
\end{equation}
a space of bounded  analytic functions defined in $\BB$ with vanishing derivatives at the origin up to order $d$. We endow these spaces with the norm \eqref{analyticnorm}, which turns $\mathcal{A}_{\delta,d}$ into a complex Banach space. Equivalently, the norm $\| K\|_{\A_{\delta,d}}$ can be defined as the smallest constant $C\geq0$ for which $|K(z)| \leq C |z|^{d}$.\\
\begin{remark}
	The proof that  the space $\mathcal{A}_{\delta,d}$ is complete under \eqref{analyticnorm} is included for completeness. We argue that, given a Cauchy sequence  $\{K_n\}_{n\in\nat}$ in $\mathcal{A}_{\delta,d}$, it is also a Cauchy sequence in $C^0$ and, by the completeness in of $C^0$ it has a $C^0$-limit, which we denote as $K$. This limit will be analytic because the uniform limit of analytic functions is analytic.  Moreover, since $|K_n(x) | \le C |x|^{d}$, we conclude that $K$ is $\mathcal{A}_{\delta, d}$. Because $K_n$ is Cauchy, we know that given $\eps > 0 $ we can find $N_0(\eps)\in\mathbb{N}$ so that if $n, m >  N_0$, then $| K_n(x) - K_m(x) \le \eps |x|^d$. Taking limits in $m$ to $\infty$, we conclude  that for $n > N_0$, we have $| K_n(x) - K(x) | \le \eps |x|^d$.
\end{remark}
Since we are interested in real-valued parametrizations of the invariant manifolds, we define \begin{equation}\label{defAreal}
\mathcal{A}_{\delta,d}^{real}=\{K\in\mathcal{A}_{d,\delta}: K \text{ takes real values for real arguments} \},
\end{equation}
which defines a linear subspace of the space $\mathcal{A}_{\delta,d}$. Since $\mathcal{A}_{\delta,d}^{real}$ is a closed linear subspace of $\mathcal{A}_{\delta,d}$, it is a Banach space as well.

\begin{remark}[Contraction properties of composition]
Weighted norms such as \eqref{analyticnorm} have been found useful in proofs dealing with weak contractions, cf. \cite{CABRE2005444,Cab2003,de1997invariant}. The relevant property of these weighted norms is that if $s: \BB\rightarrow \BB$, $s(0) =0$ is a contraction, i.e., $\Lip(s) < 1$, the operator $K\rightarrow K\circ s$ is an even stronger contraction in such norms. Indeed, 
\begin{equation}
|K\circ s (z) |  \le |s(z)|^{d} \| K\|_{\A_{\delta,d}} \leq
 \Lip(s)^{d}|z|^{d}  \| K\|_{\A_{\delta,d}}.
\end{equation}
Hence, if $s$ is a contraction fixing the origin, we obtain 
\begin{equation} \label{goodcontraction}
\| K\circ s \|_{\A_{\delta,d}}  \leq  \Lip(s)^{d} \| K \|_{\A_{\delta,d}}.
\end{equation}
If $\{K^n_\eps(z)\}_{n\in\nat}$ is a sequence of functions analytic jointly in $(z,\eps)$ and in $z$ for fixed $\eps$, and the sequence $\{K^n_\eps(z)\}_{n\in\nat}$ converges in $\A_{\delta,d}$ to $K_\eps$, then $K_\eps$ is also jointly analytic in $(z,\eps)$, cf. \cite[Chapter III]{hille1996functional}.
\end{remark}

\section{Lyapunov Subcenter Manifolds}
In this section, we review some classical results on the theory of Lyapunov subcenter manifolds. For more details and other variants of the results, we refer the reader to \cite{liapounoff1907probleme, MoserZ, MeyerH, Duistermaat1972,kelley1967liapounov,kelley1969analytic,moser1976periodic}.\\
Specifically, we consider a differential equations of the form
\begin{equation}
\dot{X}=LX+N(X)+ \eps CX + \eps G_\eps(X) \equiv F_\eps(X) ,
\label{main}
\end{equation}
for an unknown $X:(0,\infty)\to\real^n, t\mapsto X(t)$ and $\eps\geq 0$.\\ The $n\times n$-matrix $L$ and the analytic nonlinearity $N(X)= O(|X|^2)$ constitute the unperturbed system, while the $n\times n$ matrix $C$ and the analytic nonlinearity $G_\eps(X)=\mathcal{O}(|X|^2)$ constitute the perturbation, i.e., we regard equation \eqref{main} as a perturbation of the system 
\begin{equation}
\dot{X}=LX+N(X),\label{unpert}
\end{equation}
on which we make the following assumptions.
\begin{assumption}\label{AssLSC}

\noindent
\begin{enumerate}
\item 
The matrix $L$ is semi-simple (i.e. diagonalizable)
\item 
The matrix $L$ has a pair of complex conjugate
eigenvalues with zero real-part, i.e.,
\begin{equation}
\{\pm\ri\omega_0\}\subset\sigma(L),
\end{equation}
for some $\omega_0>0$. 
\item The remaining $n-2$ eigenvalues of $L$, which we denote by $\{\mu_{k}\}_{1\leq k\leq n-2}$,
are \textit{non-resonant} with the eigenvalues $\pm\ri\omega_0$, i.e.,
\begin{equation}\label{nonresunpert}
\frac{\mu_{k}}{\ri\omega_0}\notin\integer,
\end{equation}
for all $1\leq k\leq n-2$. In particular, $0\notin\spec(L)$. 
\item There exists an analytic  first integral to
equation \eqref{unpert}, i.e., there exists an analytic 
 $I:\real^n\to\real$
such that for any solution $t\mapsto X(t)$, 
\begin{equation}
\frac{d}{dt}I(X(t))=0.\label{FirstInt}
\end{equation}

We will assume  that the function $I$ is normalized to $I(0)=0$ (without loss of generality), satisfies $\nabla I(0)=0$ and its second derivative at the origin is non-degenerate (without loss of generality, positive definite) on the eigenspace associated with $\pm\ri\omega_0$, i.e., 
\begin{equation}
D^{2}I(0)(Y,Y)>0,  \label{posdef}
\end{equation}
for all $Y\in\mathbb{R}^n, Y \in \eig(\pm \ri \omega_0)$.
\end{enumerate}
\end{assumption}

\begin{remark}
Because of assumption \eqref{posdef}, the energy $I$ is equivalent to $|x|^2$ near the origin in the LSM. Also, the orbits of the unperturbed system \eqref{unpert} stay on level set of $I$ by assumption \eqref{FirstInt}. Therefore, inside the LSM, we can define \textit{action-angle coordinates}, which are geometrically equivalent to $\Big(|x|^2, \text{Arg}(x)\Big)$, cf. \cite{arnol2013mathematical}. In these coordinates, the orbits of the unperturbed system are just circles.	
\end{remark}

\begin{remark} 
The assumption that $L$ is semi-simple will not play an important role. It is not used in the Lyapunov subcenter theorem nor on the existence of asymptotic expansions. For the proof of existence of spectral submanifolds, only two consequences are used, namely, the analytic dependence of the dissipation parameter $\varepsilon$ and the persistence of certain non-resonance conditions. If they can be verified by other means, we do not need semisimplicity of $L$.
\end{remark} 

\begin{remark} 
For the Lyapunov subcenter theorem, we do not need to impose any restriction on the eigenvalues $\mu_k$ except the nonresonance. They could  be imaginary or have non-zero real part. \footnote{Note, however, that the preservation of $I$ imposes some restrictions. If some eigenvalues have positive or negative real parts, the conserved quantity has to be degenerate along the eigenspaces related to these eigenvalues.}
As for our results, the (formal) expansions will not require any restrictions on $\mu_k$ beyond the non-resonance with $i\omega_0$. The results on convergence, i.e., on the existence of a true solution to the invariance equation, presented in this paper will require that the $\mu_k$ are imaginary as well as a further non-resonance condition and the assumption that the unperturbed system is Hamiltonian.\\ Note that, for a general system, we can always reduce to the case of imaginary eigenvalues by taking a restriction to the center manifold. Using the center manifold reduction, however, requires dealing with the problem of non-uniqueness of the center manifold and that it is only finitely differentiable. These problems will require different techniques. We hope to come back to them in a subsequent paper.
\end{remark}

\begin{remark}
In our convergence results we will also assume that the system is Hamiltonian. This is a natural assumption for the applications to mechanical systems. From the mathematical point of view, this leads to some uniform estimates. See Lemma~\ref{uniform}. In Example~\ref{non-uniform}, we show that in the case that the system is energy preserving but not Hamiltonian, the uniform estimates in Lemma~\ref{uniform} may be false. It seems that, in this case, the results of convergence may be false and that there are new phenomena that may appear. Again, dealing with these new phenomena will require new techniques. 
\end{remark}

We let 
\begin{equation}
\begin{split} & X_{1}:=\eig(\pm\omega_0\ri),\\
 & X_{2}:=\eig\Big(\{\mu_{k}\}_{1\leq k\leq n-2}\Big).
\end{split}
\end{equation}
Using the spectral projections, we can decompose the phase space as $X=X_{1}\oplus X_{2}$. The linear spaces $X_{1}$ and $X_{2}$ are invariant under $L$, i.e., $L(X_{1})\subseteq X_{1}$ and $L(X_{2})\subseteq X_{2}$. We let $\pi_{X_{1}}:X\to X_{1}$ and $\pi_{X_{2}}:X\to X_{2}$ be the corresponding projections (cf. Figure \ref{ImgLSC}) and define
\begin{equation}
L_{1}:=\pi_{X_{1}}L,\quad L_{2}:=\pi_{X_{2}}L.
\end{equation}
For later computations, we will denote the variables in the spaces $X_{1}$ and $X_{2}$ as 
\begin{equation}
(x,y)\in X_{1}\oplus X_{2}=\real^n,
\end{equation}
with $\dim(x)=2$ and $\dim(y)=n-2$.\\
The following theorem summarizes the classical existence results on a one-parameter family of periodic solutions close to the trivial solution $X=0$ of equation \eqref{unpert}.

\begin{theorem}\label{exLSC} 
Let system \eqref{unpert} satisfy parts (2), (3), (4) of Assumption \ref{AssLSC}. Then, there exists a two-dimensional, invariant manifold
$M_{0}$ tangent to the spectral subspace $X_{1}$. The manifold $M_{0}$ is analytic and filled with a one-parameter family of periodic orbits.
\end{theorem}

We refer to the invariant manifold $M_{0}$ as
a \textit{Lyapunov subcenter manifold} (LSM). There exists a constant $\delta>0$, such that, locally around the origin, we can describe this LSM as the graph of an analytic function $w_{0}:B_{\delta}^{2}\to\real^{n-2}$, as illustrated in Figure \ref{ImgLSC}.\\
Note that the fact that an orbit is periodic is a topological property, so that the set is unique under topological properties. 

\begin{figure}[h]
\centering \includegraphics[scale=0.4]{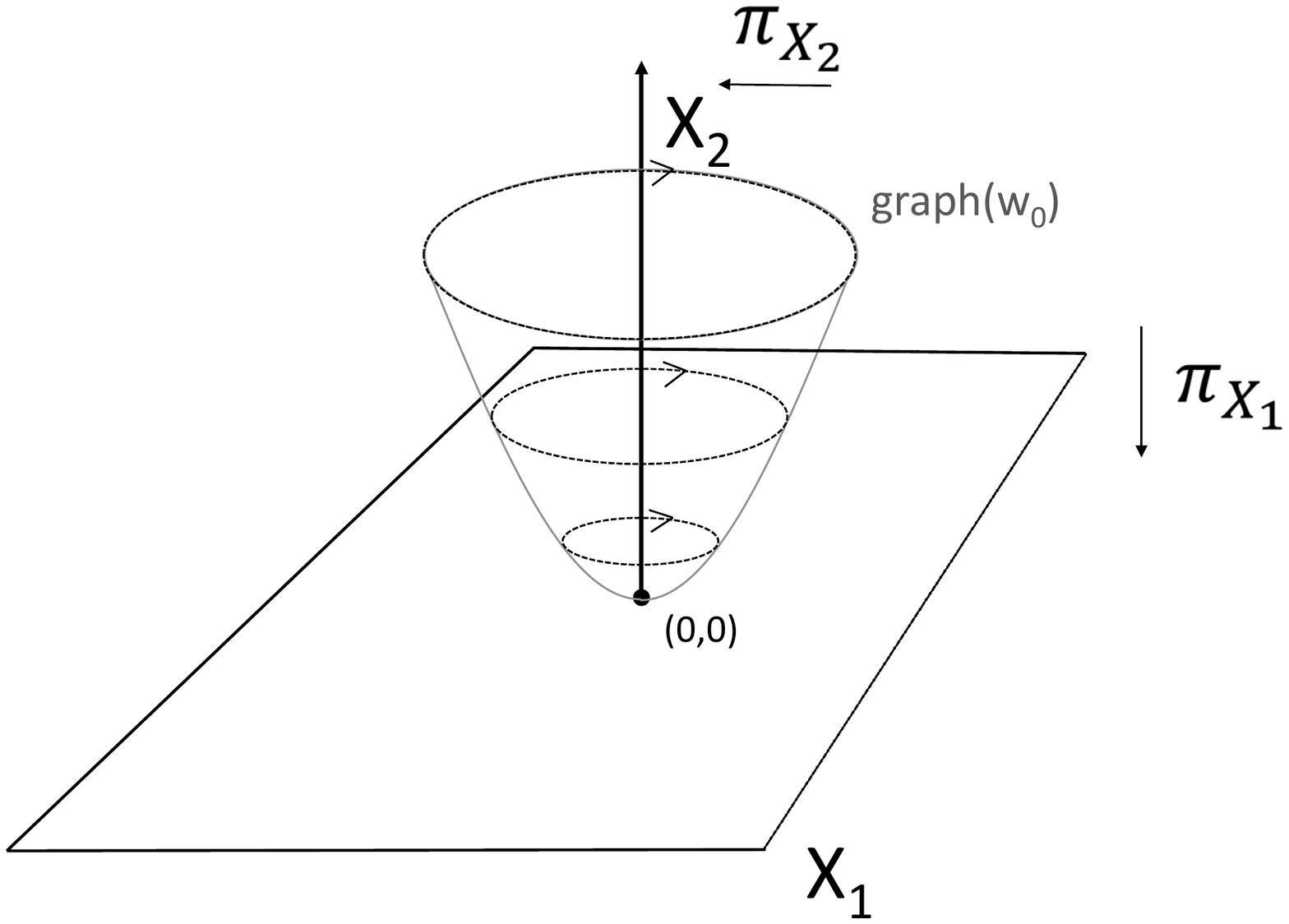}
\caption{LSM tangent to the two-dimensional spectral subspace $X_{1}$,
represented
as the graph of an analytic function $x\protect\mapsto w_{0}(x)$
with radius of convergence $\delta$. The manifold is unique and filled
with periodic orbits.}
\label{ImgLSC} 
\end{figure}

\begin{proof} 
This result was first proved in 
\cite[p. 352]{liapounoff1907probleme}. See also
\cite{kelley1967liapounov,MoserZ, MeyerH} for similar arguments as the argument presented here. This  proof does not assume that the 
system is Hamiltonian (only  that it  has a conserved quantity) and it allows that the system has repeated eigenvalues or Jordan blocks, some of which could be stable/unstable.\\
We introduce a small parameter $\nu$ and scale the variables in \eqref{unpert}. Writing $x = \nu u$ we, see that \eqref{unpert} is equivalent to 
\begin{equation}\label{scaled} 
u' = L u + \nu^{-1} N(\nu u).
\end{equation} 
Because $N$ vanishes up to second order, we have $\nu^{-1}N(\nu u) = \nu \tilde N_\nu(u) $, 
so that \eqref{scaled} has a well defined limit as $\nu$ tends to $0$. We also observe that if \eqref{unpert} preserves $I$, then \eqref{scaled} preserves $I_\nu(u) = \nu^{-2} I (\nu u)$. Note that $I_\nu$ has a well defined limit $I_0(u) = D^2I(0)(u,u)$.\\
We start by  studying the limit $\nu = 0 $ of \eqref{unpert}. In the two dimensional space $X^1$, the flow is just a rigid rotation with period $T \equiv 2 \pi/\omega_0$. We note that the spectrum of $\exp( TL)$ is the exponential of the spectrum of $T L$, that is 
\begin{equation} \label{eigenvalues}
\sigma( \exp{TL})  = 
\left\{ \{\exp( 2 \pi \ri \mu_k/\omega_0)\}_{k= 1}^{n-2}, 1, 1\right\}.
\end{equation}
We write the two dimensional real plane  $X_1$ corresponding to the eigenvalues $\pm \ri \omega_0$   as  the set of points $(x_1, x_2)$ and we will denote the points in the complementary spectral space as $y$.\\
If we consider the return map $R_{0,E}$ of
to the co-dimension one plane  $x_1 = 0$ restricted to a level surface of the conserved quantity\footnote{
We will refer to this conserved quantity as the energy since this is what happens in many problems. It also allows to use  names such as 
``energy surface'' for the level sets etc.} we see that the spectrum of the return map restricted to an energy surface, cf.  \eqref{eigenvalues}, is just 
$\{\exp( 2 \pi \ri \mu_k/\omega_0)\}_{k= 1}^{n-2}$ since the two eigenvalues $1$ of  $\exp{LT}$ correspond respectively to the translation along the flow (eliminated by the return map) and the translation along the energy (eliminated by taking the energy surface). The non-resonance assumption of the theorem tells that $\mu_k/\omega_0$ is not an integer, hence, the eigenvalues of the return map restricted to the energy surface  are not $1$. Now we observe that $R_{\nu,E}$ depends analytically on $\nu$ for $\nu$ small.\\
The previous observations amount to the fact that writing points in the axis $x_1 = 0$
say that $R_{0, E}(0, x_2(E), 0)  = (0, x_2(E), 0)$. Furthermore, $\partial_y R_{0, E}(0,x_2(E), y) |_{y = 0}  - \Id$ is an invertible matrix. Hence, applying the finite dimensional, implicit function theorem \cite{Lang1999, Dieudonne,Meyerimplicit} we obtain that, for small $\nu$, we can find analytic families of fixed-points of the return map $R_{\nu, E}$ indexed by $\nu$ and $E$. This argument also shows that the manifold is analytic everywhere, including at zero. Indeed, it is an interesting exercise to compute the coefficients of the expansion at zero of the manifold.\\
Using the scaling, it is not difficult to show that the family is tangent at zero to the eigenspace. One can also observe that the implicit function theorem allows to compute the derivatives of the manifold at the origin. We leave the details to the reader,  see also \cite{Moulton20}.
Also the local uniqueness statements are those of the standard implicit function theorem. The periodic orbits are locally unique in the energy surface.
\end{proof}

\begin{remark}
The method of proof of Lyapunov center theorem presented here has been generalized   to infinite dimensional systems \cite{Bambusi00} or systems with symmetry in \cite{BuonoLM05,CallejaDG18}. It would be interesting to study the effects of adding dissipation to these models. 
\end{remark} 

\begin{remark} 
Since the proof above is based only on the implicit function theorem, it applies also to finitely differentiable systems. If the vector field is $C^\ell$, $\ell \ge 1$,  we  get a $C^\ell$ manifold, see \cite{kelley1967liapounov}. 
\end{remark}


\begin{remark}
To set up the analyticity results, it is convenient to examine what happens for complex values of the variables and the parameters.\\
We observe that if we consider now $x \in \complex^2$, the scaling arguments still work and the flow is transversal to $x_1 = 0$ in the complex sense, entailing the return time to be a complex variable. The periodic orbit will consist of the orbit for complex times in a neighborhood of the path joining $0$ to the complex period. The union of all these periodic orbits covers $\BB$.\\
The singular nature of the dissipative perturbations, is also apparent in the complex interpretation. Once we add a dissipation, using Sternberg theorem \cite{sternberg1957local},
we know that the exponentially contracting orbit conjugate to an exponential, hence a complex periodic orbit. So, the family of 
periodic orbits in the conservative system bifurcate into a single complex periodic orbit. This clearly indicates the singular nature of the problem.
\end{remark}

For later calculations, we introduce a normalization of the vector field $F_0(X)$ such that the dynamics on the LSM are just given by constant-phase rotations. This will simplify subsequent arguments.\\
Let $T=T(I)$, only depending upon the energy, be the first return time of the periodic orbit with energy $I$ to a line of section and define $\Omega(I)=\frac{2\pi}{T(I)}$. Then, the dynamics on the LSM associated to the system
\begin{equation}\label{suspension}
\dot{X}=\frac{\omega_0}{\Omega(I)}F_0(X),
\end{equation}
is just given by rigid rotations with frequency $\omega_0$.\\
Note that, multiplying a vector field by an scalar, does not change the invariant manifolds, the periodic orbits or the conserved quantity. In the Hamiltonian case, if we multiply the vector field by a function of the Hamiltonian, we obtain a Hamiltonian vector field.\footnote{If $X = J \nabla H$, then,
for any function $\alpha: \real\rightarrow \real$, we have $\alpha(H) X = J \alpha(H)\nabla H  = J \nabla \beta(H)$ where $\beta' = \alpha$. Hence, the time-scaled vector field is also Hamiltonian}

\begin{remark}
The advantage of multiplying the vector field is that it is obvious that the derivative of the flow restricted to the Lyapunov manifold is just a rotation (hence modulus $1$).  This, of course, could be obtained also by defining a new system of coordinates. \\
The normalization \eqref{suspension} is an explicit application of the standard suspension construction  explained e.g. in \cite{katokh}, showing that, for a compact manifold, the special flow with respect to the first return time is equivalent to the suspension flow. The construction presented makes the constructions more explicit.
\end{remark}

We will need the following lemma in later calculations when perturbing from the LSM.   

\begin{lemma}\label{uniform}
Assume that system \eqref{unpert}, normalized according to \eqref{suspension}, satisfies Assumptions \eqref{AssLSC}. Choose coordinates $(x,y)$, $x\in\mathbb{R}^2$ and $y\in\mathbb{R}^{n-2}$, such that the LSM corresponds to the plane $\{y=0\}$ and let
$\phi=\phi^{T_0}$ be the time-$T_0$ map of system 
\eqref{suspension}, for $T_0=\frac{2\pi}{\omega_0}$.
Then we have
\begin{equation}\label{time1y}
\phi(x,y)=(x+B(x)y,A(x)y)+\mathcal{O}(|y|^2),
\end{equation}
for some $n-2\times n-2$-dimensional matrix function $A$ and some $2\times n-2$-dimensional matrix function $B$ depending only on the energy.\\
In case that the flow is Hamiltonian and that $A(0)$ has only simple eigenvalues with modulus one, we have that the eigenvalues of the matrix $A(x)$ are of modulus one for all $|x| \in \mathbb{R}^2$ with $|x| \le  \delta$, for some $\delta>0$\footnote{This $\delta$, for which the coordinates \eqref{time1y} with all eigenvalues of $A$ having modulus one, will be the fundamental domain of existence for the later perturbation argument.}.
\end{lemma}

\begin{proof}
We can take coordinates in which the LSM corresponds to $\{ y= 0\}$. After the normalization \eqref{suspension}, all the periodic orbits have period $T_0$.
By assumption, the Jacobian $D\phi(x,0)$ is a symplectic matrix. It is well known that for symplectic matrices, the inverse of the eigenvalues are also eigenvalues, since, if a symplectic matrix has simple eigenvalues in the unit circle, all the nearby matrices have an eigenvalue in the unit circle too.\\
Our assumption  that $A(0)$ has simple eigenvalues implies that $D\phi(0,0)$ has a double eigenvalue $1$ and all the other eigenvalues are simple in the unit circle. Hence, for small enough $x$, we see that $D\phi(x,0)$ has to have $n-2$ eigenvalues on the unit circle. The double eigenvalue $1$ could, in principle bifurcate, but it does not because of the invariance of the LSM and the conservation of energy inside of the LSM. 
\end{proof}

\section{The main theorem} 
\label{sec:statement}

\subsection{An analytical formulation of the problem}

In this section, we  translate the geometric problem of invariant manifolds into a functional analysis problem by following the idea of the parameterization method \cite{Cab2003,CABRE2005444,haro2016parameterization}. Given a vector field $F_\eps$ (satisfying the hypothesis of the subsequent theorem) we will seek an embedding $K_\eps: B_\delta^2 \rightarrow \real^n$ (which extends to an embedding defined on $\tilde B_\delta^2$)  and another vector field $R_\eps:B_\delta^2 \rightarrow \real^2$ (which also extends to $\tilde B^2_\delta$) in such a way that
\begin{equation}\label{CH}
F_\eps(K_\eps(x))  = DK_\eps(x)R_\eps(x),
\end{equation} 
with $K_\eps(0) = 0$ and $R_\eps(0) = 0$.\\
The equation \eqref{CH} will be the centerpiece of our analysis. 
Note that the geometric meaning is that the range of $K_\eps$ is invariant under the flow of $X_\varepsilon$, i.e., the vector field $F_\eps$ at one point in the range is tangent to the range. The vector field $R_\varepsilon$ is then a representation of the dynamics on the manifold.  

\begin{remark} 
Since $K_\eps$ is an embedding, it can be used to follow several turns of the manifold which are very different from being a graph. There are numerical examples \cite{haro2016parameterization,kalies2018analytic,van2016parameterization} in which the same parameterization can be used to follow a large area containing turns and folds of the manifold in some model examples such as the Lorenz equations. The fact that the proofs are based on a contraction mapping argument allows to justify rigorously any method that produces approximate solutions, e.g., numerical computations. Using the contraction mapping theorem, we can show that if some function moves a very small amount by the application of the operator, then there is a fixed point at a distance from the approximate solution comparable to the distance of the approximate solution to its iterate. This allows to validate numerical calculations rigorously.
\end{remark}

\begin{remark} \label{underdetermined} 
Equations \eqref{CH} are highly under-determined. We have already remarked in \eqref{suspension} that we can change the time multiplying the vector field by a scalar function without affecting the invariant manifolds or the conserved quantities, but there are other sources of undeterminacy as well. In fact, any change  of variables in the reference disk leads to a parameterization of the same manifold. Indeed, if $K_\eps, R_\eps$ are a solution of \eqref{CH} and $h_\eps$ is a local diffeomorphism $h_\eps(0) = 0$, we have 
\begin{equation}
\begin{split} 
F_\eps \circ K_\eps \circ h_\eps &=  (DK_\eps  R_\eps) \circ h_\eps  = 
D(K_\eps)  \circ h_\eps \,  R_\eps \circ h_\eps  \\ 
&= D( K_\eps  \circ h_\eps) (Dh_\eps)^{-1}   R_\eps \circ h_\eps
\end{split} 
\end{equation}
In other words, $\tilde K_\eps =K_\eps \circ h_\eps$, 
$\tilde R_\eps = (Dh_\eps)^{-1}   R_\eps \circ h_\eps$ is also a solution of \eqref{CH}.
One can show, however, that, up to this family of transformations, the manifold is unique among the $d$-times differentiable ones, where $d$ is a number that depends on the spectral properties of $DF_\varepsilon(0)$,  cf. \cite{de1997invariant,Cab2003,CABRE2005444}.\\
We will take advantage of this underdeterminacy to impose some normalization conditions on the parametrization $K_\eps$ and the vector field on the manifold $R_\eps$. From the computational point of view, the underdeterminacy  of equation \eqref{CH} can be used to construct more efficient algorithms for the computation of invariant manifolds as well, cf. \cite{haro2016parameterization}. Furthermore, \cite{sternberg1957local} shows that, in the absence of resonances, there is a system of coordinates in which $R_\eps$ is linear. In our situation, however, this cannot be achieved due to the singular nature of the problem. A linear vector field on the invariant manifold, as described in \cite{sternberg1957local}, would lead to singularities in the expansions as $\eps\to 0$.  By allowing higher order $z$-terms in $R_\eps$ we can avoid the singularities of the change of variables leading to the linearization.
\end{remark} 

We will make the following assumption on the linear part of the perturbed system:
\begin{assumption}\label{Asspert}
The matrix $L+\eps C$ has a pair  of complex conjugate eigenvalues 
$\lambda^{\pm}_\eps$, perturbing from the nonresonant eigenvalues in Assumption \ref{AssLSC}, such that
\begin{equation}
\lambda^\pm_\eps =
 \pm\ri\omega_0 + (- \alpha  \pm \ri \alpha_I) \eps  +  O(\eps^2),
\end{equation}
for some $\alpha>0$ and $\alpha_I\in\mathbb{R}$.
\end{assumption}

\begin{remark}
We note that the eigenvalues depend differentiably on $\eps$ at $\eps = 0$ as a consequence of the nonresonance condition. \\
If $\alpha < 0$, we obtain similar results by switching the direction of time. The content of Assumption~\ref{Asspert} is that $\alpha \ne   0$. We will see that if $\alpha = 0$, the conclusions of the main theorem may be false and there may fail to be an SSM of size 
one in an neighborhood, see Example~\ref{firstexample}. On the other hand, the quantity $\alpha_I$, i.e., the change of the (pseudo)-frequency induced by the dissipation, does not play any role in our analysis.\\
We also note that Assumption~\ref{Asspert} is a condition on the perturbation, not on the unperturbed problem which is very typical for singular perturbation problems.\\
In practical problems, verifying Assumption~\ref{Asspert} is an easy task, since it only involves checking the first order perturbation theory for eigenvalues of a finite dimensional matrix. 
\end{remark}

\subsubsection{Contraction properties of the  perturbed linear part}
\label{pertlin}

Let $X_{1}^{\eps}$ be the eigenspace of $L+\eps C$ that perturbs from $X_{1}$, i.e., $X_{1}^{\eps}\to X_{1}$ as $\eps\to 0$, and let $X_{2}^{\eps}$ be its spectral complement. Since $L+\eps C$ is semi-simple by assumption for $\eps$ small enough, we again have that $X=X_{1}^{\eps}\oplus X_{2}^{\eps}$. Since $L$ does not have repeated eigenvalues, the remaining eigenvalues $\{\mu_k\}_{1\leq k\leq n-2}$ continue to a family of eigenvalues 
$\{\mu^\eps_k\}_{1\leq k\leq n-2}$, which is differentiable at $\eps = 0$ and the corresponding spectral subspace and the spectral projection are differentiable in $\eps$, satisfying
\begin{equation}\label{Oepsspace}
\begin{split}
&X_1^\eps=X_1+\mathcal{O}(\eps),\qquad X_2^\eps=X_2+\mathcal{O}(\eps),\\
& \pi_{X_1^\eps}=\pi_{X_1}+\mathcal{O}(\eps),\qquad \pi_{X_2^\eps}=\pi_{X_2}+\mathcal{O}(\eps),
\end{split}
\end{equation}
where $X_1^\eps=\eig(\lambda^{\pm}_\eps)$, for $\eps$ small enough. All this follows from spectral perturbation theory for matrices, cf. \cite{kato1995perturbation}, \cite[p. 396 Theorem 1]{lancaster1985theory}. Indeed, due to the nonresonance condition \eqref{nonres}, the
eigenvalues $\lambda_\eps^{\pm}$ necessarily have algebraic multiplicity one, which then implies the $\mathcal{O}(\eps)$-dependence in $X_1^\eps$.

\begin{remark}
Generally, if a matrix $L$ has a repeated eigenvalue $\lambda$, only the weaker relation
\begin{equation}
\eig(\lambda_\eps)=\eig(\lambda)+\mathcal{O}(\eps^{1/p}),
\end{equation}
for some $p>0$, holds, cf. \cite{lancaster1985theory}, p.402, Theorem 1. Here, $p$ is the length of the Jordan block associated with $\lambda$. Indeed, e.g., the matrix
\begin{equation}
\left(\begin{matrix}
1 & 1\\ \eps & 1
\end{matrix}\right),
\end{equation}
has spectrum $\{1\pm\sqrt{\eps}\}$ and the corresponding eigenspaces are spanned by the vectors $(1, \pm \sqrt{\eps})$. Note that the eigenvalues move $\sqrt \eps$ in this case, i.e., much faster than $\eps$ and that the angles between the eigenspaces are also small. The later may cause problems if we need to consider the projections over these spaces, since they will have a norm that grows as $\eps$ goes to zero.
\end{remark}

Let $\beta$ be defined as in Lemma \ref{Dyeps} and let $\alpha$ be the linear contraction rate as defined in Assumption \ref{Asspert}. Then, there exists a number $d$ such that
\begin{equation}\label{defL}
\beta<d\alpha.
\end{equation}
Note that if the condition \eqref{defL} is satisfied for some $d$, it is satisfied for all larger ones. Any of those will work for our subsequent considerations. We will choose the parameter $d$ in the definition of the space $\mathcal{A}_{\delta,d}$ such that  \eqref{defL} is satisfied, i.e., depending upon the ratio between the linear contraction rate on the parametrization space of the LSM and a parameter depending on the second derivative of the flow map.\\

\subsection{Formulation of the main theorem}

In this section, we introduce some fundamental parameters that and formulate our main theorem. For the formulation and the proof, we need the following lemma on the time-$T_0$ map for $y$-values of order $\varepsilon$.
\begin{lemma}\label{Dyeps}
Assume that system \eqref{unpert}, normalized according to \eqref{suspension}, satisfies Assumptions \eqref{AssLSC} and assume that the flow is Hamiltonian, such that the normal form \eqref{time1y} holds for all $|x|<\delta$. Let $\phi_\varepsilon^t$ be the flow map of the perturbed system. Then, the Jacobian of the inverse of the time-$T_0$ map of the perturbed system \eqref{main}, which we denote as $\phi_\varepsilon^{-1}$, satisfies \begin{equation}\label{estDphi-1}
\|D\phi^{-1}_\varepsilon(x,\varepsilon y)\|=1+\beta\varepsilon+\mathcal{O}(\varepsilon^2),
\end{equation}
for $\varepsilon>0$, for all $|x|<\delta$, $|y|<\eta$, where $\eta>0$ independent on $\varepsilon$. Here, $\beta$ is such that
\begin{equation}
\|D_yD\phi_0(x,0)\|\leq \beta,
\end{equation}
for all $|x|<\delta$.
\end{lemma}
\begin{proof}
The proof is an immediate consequence of Lemma \ref{uniform}. Indeed, thanks to the form of $\phi$ in \eqref{time1y} and the simplicity of the eigenvalues of $A$ for small enough $x$ (non of which is equal to one), we can find matrices $Q$, depending on $x$, such that $D\phi$ can be block-diagonalized as
\begin{equation}
D\phi(x,y)=\left(\begin{matrix}
1 & Q(x)\\ 0 & 1
\end{matrix}\right)\left(\begin{matrix}
1 & 0\\ 0 & A(x)
\end{matrix}\right)\left(\begin{matrix}
1 & Q(x)\\ 0 & 1
\end{matrix}\right)^{-1}+\mathcal{O}(|y|),
\end{equation}
just by choosing $Q(x):=(A-1)^{-1}B$. Therefore, since all eigenvalues of $A$ are simple for $|x|<\delta$, it follows from Taylor-expanding $D\phi_\varepsilon$ in $y$ that
\begin{equation}
\|D\phi_\varepsilon(x,\varepsilon y)\|=1+\beta\varepsilon+\mathcal{O}(\varepsilon^2),
\end{equation}
for all $|x|<\delta$ and $|y|<\eta$, for some $\eta$ independent on $\varepsilon$.  By the standard implicit function theorem, the claim follows. 
\end{proof}
	The set up for our main theorem involves four small parameters:
	\begin{itemize}
		\item The parameter $\delta$, which controls the domain of definition of the LSM (and also the domain of definition of the SSM in the perturbation). This parameter will be independent on the dissipation $\varepsilon$.
		\item The parameter $\th$, which controls the aperture of the cone of complex values for $\eps$, domain considered.
		\item The parameter $\tau$ which controls  the size of the complex 
		extension.
		\item The parameter $\eps_0$, which is the maximum value of  $|\eps|$ (as a complex number) for which the results are valid.
	\end{itemize}
	Along the proof, we will specify some smallness conditions on these quantities. 
	It will be important that all these parameters (in particular $\delta$) can be chosen uniformly in the value of the dissipation, $\varepsilon$, so that we obtain results for $\delta$, $\theta$ and $\tau$, which are uniform in $\varepsilon$.\\
We are now ready to formulate informally our main result. 
\begin{theorem}\label{mainThm} Consider the system 
\begin{equation}\label{maininthm}
\dot{X}=LX+N(X)+ \eps CX + \eps G_\eps(X) \equiv F_\eps(X),
\end{equation}
under Assumptions \ref{AssLSC} and Assumption \ref{Asspert}. Assume further that the unperturbed system is Hamiltonian, that the matrix $L$ does not have repeated eigenvalues and that all the eigenvalues are imaginary.\\
Then, for  $\eps$ sufficiently small, there exists an invariant, analytic, two-dimensional manifold $M_{\eps}$ around the origin for the perturbed system \eqref{maininthm}. As $\eps$ converges to zero, the manifold $M_\eps$ converges -- in the sense of analytic manifolds -- to the LSM (of the unperturbed system) in a domain which is independent of $\eps$. Moreover, we have explicit asymptotic expansions to any order in $\eps$, uniformly valid in an $\varepsilon$-independent domain. The manifold is unique among the invariant manifolds that are sufficiently differentiable at the origin. 
\end{theorem}

The precise formulation 
of the main result is the following Theorem~\ref{mainprop}.

\begin{theorem}\label{mainprop} 
Consider the system 
\begin{equation}\label{maininthm2}
\dot{X}=LX+N(X)+ \eps CX + \eps G_\eps(X) \equiv F_\eps(X) ,
\end{equation}
under Assumptions \ref{AssLSC} and Assumption \ref{Asspert}.  Assume further that the unperturbed system is Hamiltonian, that the matrix $L$ does not have repeated eigenvalues and that all the eigenvalues are imaginary satisfying
\begin{equation}\label{no-resonances}
\mu_k - \mu_l  \ne m \, 2 \pi \ri \omega_0    \quad \forall k,l = 1,\ldots n-2, k \ne l \forall m \in \mathbb{Z}.
\end{equation}

Let $\delta$ be the maximal radius for which the flow map of the unperturbed part of system \eqref{maininthm} can be written according to Lemma \ref{uniform} and let $\beta$ be the minimal upper bound of the inequality
\begin{equation}
\|D_yD\phi_0(x,0)\|\leq \beta,
\end{equation}
for all $|x|<\delta$, as in Lemma \ref{Dyeps}. Choose any $d$ that satisfies 
\begin{equation}
\beta<d\alpha,
\end{equation}
where $\alpha$ is as in Assumption \ref{Asspert}.\\

Then, there is a sequence of functions 
$K_j: B_\delta \rightarrow \complex^n$ (extending to functions from $\BB$) and 
$R_j: B_\delta \rightarrow B_\delta$ (extending also to functions from $\BB$), such that $K_j(0)=0$, $R_j(0)=0$, $DK_0(0)$ being the embedding from $\real^2$ to the real part  $X^1_\eps$ and such that 
\begin{equation}
DR_0(0)  = \begin{pmatrix} 0 & \omega_0  \\ -\omega_0   &  0\end{pmatrix},
\end{equation}
with the following properties:
\begin{enumerate}
\item The sequences $K_j$ and $R_j$ solve the equation \eqref{CH} in the sense of 
formal power series, i.e., for any $N \in \nat$, setting 
\begin{equation}
\begin{split} 
K^{\le N}_\eps(z)=\sum_{j=0}^N K_j(z) \eps^j,\quad R^{\le N}_\eps(z)  =   \sum_{j=0}^N R_j(z) \eps^j,
\end{split}
\end{equation}
we have that
\begin{equation}\label{firstconclusion} 
\| F_\eps(K_\eps^{\le N}(\cdot ))  -  DK_\eps^{\le N} (\cdot)R_\eps^{\le N}(\cdot) \|_{\A_{\delta,d}} 
\le C_N |\eps|^{N+1},
\end{equation} 
for all $\eps$ small enough.
\item
For $N \ge d$, where again $d$ is as in \eqref{defL}, there exists a unique 
$K_\eps \in \mathcal{A}^{real}_{\delta, d}$ such that
\begin{equation}
 F_\eps(K_\eps(z)) =   DK_\eps(z)R^{\le N}_\eps(z)
\end{equation}
for all $\eps \in \C_\th$ with $\varepsilon$ and $\th$ small enough. 
\item Furthermore, for $\eps \in \mathcal{C}_\th$ (in particularly for all $\eps >0$)  small enough, we have
\begin{equation}\label{secondconclusion} 
\| K^{\le N}_\eps - K_\eps \|_{\A_{\delta,d}} \le C |\eps|^{N},
\end{equation} 
for some constant $C>0$. 
\end{enumerate}

\end{theorem}

\begin{remark}
The polynomials $K^{\le N}_\eps$ and $R^{\le N}_\eps$, satisfying the invariance equation up to order $N$ in $\varepsilon$, i.e., with a small error, will be obtained through formal calculations. The approximate solutions will satisfy the invariance equations for all $\eps$, small enough, in a complex ball. In fact, the Assumption that the system is Hamiltonian is not needed for the formal calculations. If we want, however, to have formal solutions that also satisfy the invariance equation to a sufficiently high order in $x$, we have to assume a normal form of the kind \eqref{normalform}, i.e., a suitable non-resonance condition.\\
The smallness condition in $|\eps|$ needed for the second conclusion in Theorem \eqref{mainprop}, however, may depend upon $N$.  This is a reasonable limitation since we do not expect that the formal sums $\hat{K}_\eps = \sum_{j=0}^\infty K_\eps^j$ and $\hat{R}_\eps =\sum_{j=0}^\infty R_\eps^j$ to converge in the sense of series of analytic functions. In practice, one can choose an $N$ which is optimal for the goals at hand.
\end{remark}

\begin{remark} 
\label{rem:negativeeps}
The approximate solutions $K^{\le N}_\eps$ and $R^{\le N}_\eps$ solve the invariance equation up to a very small error for all complex $\eps$ small, in particular, also for $\eps< 0$ and small. 
Then, these give approximately unstable manifolds. Whether these ghost manifolds correspond to actual invariant manifolds, is not obvious. In the case that the $\mu_k$ are all imaginary, by reversing the direction of time and changing the sign of $\eps$ we can can apply Theorem~\ref{mainprop} to obtain slow unstable manifolds for $\eps \in -\mathcal{C}_\th$.\\
These unstable SSMs for $\eps< 0$  are smooth continuations of the stable SSMs for $\eps>0$. They are of physical interest, for example, in systems with active media or in periodic perturbations. 
\end{remark}

For subsequent arguments, it will be important that the approximate solutions obtained in the first conclusion of Theorem \eqref{mainprop} solve equation \eqref{CH} approximately in a neighborhood of the origin of size $\delta$, which is independent of 
$\eps$. The procedure to find the solutions will be a global perturbation argument that depends on the known solutions of \eqref{CH} for $\eps = 0$ given by the LSM.\\
In a second step, we will show that these approximate solutions can be corrected to true solutions.  We take $R^{\le N}_\eps$ as the solution, but we need to correct $K^{\le N}_\eps$. Hence, in the second step, the only unknown is $K^\eps$. Again, we note that these functions have to be defined in domains whose size is uniform as the dissipation goes to zero.\\
These two steps are achieved by different methods. The calculation of the approximate solutions in the first step is done using a formal expansion based in a global averaging method. They provide approximations on a fixed neighborhood of $x$. The correction of the approximate solution into a true solution is based on transforming equation \eqref{CH} into a fixed-point problem in a small ball in an appropriately chosen function space, centered at the approximate solution.\\
The fact that we have to divide the proof into two different stages is very typical of singular perturbation theories. In the first stage, we get some perturbation that gets us a flimsy foothold in a neighborhood of the problem and then we switch to a more effective method. 
The second stage includes hypothesis on what is the outcome of the first stage.

\begin{remark}
Note that in the second step of the argument, we need some more assumptions. Notably, we use the condition \eqref{no-resonances} and the fact that the unperturbed system is Hamiltonian. This enters because in the contraction argument, we use Lemma~\ref{uniform}.
\end{remark}

\begin{remark}
The condition \eqref{no-resonances} is equivalent to assuming that $\exp( T_0 \mu_k)$ are all different complex numbers, which is one of the conditions of Lemma~\ref{uniform}. Note that the condition \eqref{no-resonances} can be verified considering only a finite number of $m$. It suffices to verify that there are no repetitions, $|m| \le \frac{1}{\omega_0}\max_{k \ne l} | \mu_k - \mu_l|$. 
 \end{remark}

\begin{remark}\label{goodanalysis} 
Both $K_\eps, R_\eps$ are unknowns of the full problem. In the first stage, we deal with both unkowns to find approximate solutions. In the second step, we take  $R_\eps = R^{\le N}_\eps$, so that the only unknown in the second state is the $K_\eps$. This is possible because we take advantage of the underdeterminacy of the equation. See Remark~\ref{underdetermined}. The fact that the second stage -- mathematically the most delicate -- has only one unknown, is an important advantage. 
\end{remark}

\begin{remark}We note that, by the non-resonance conditions, the normal forms up to order $d$ for any solution are determined. By the linearization theorem in \cite{sternberg1957local}, 
any possible $R_\varepsilon$ can be written as $R_\varepsilon=(Dh_\varepsilon)^{-1}R_\varepsilon^{\leq N}\circ h_\varepsilon$ in a neighborhood. Hence, taking advantage of \eqref{underdetermined}, it would suffice to take $R_\varepsilon^{\leq N}$ in a neighborhood. Of course, being conjugate in a neighborhood is not enough for our purposes, since we want that the flow is defined in a neighborhood uniform in $\varepsilon$ and $h_\varepsilon$ may fail to do so. The $R_\varepsilon^{\leq N}$ is a good candidate to use since it is defined in a uniform neighborhood. In summary: There is not going to be any advantage to find $R_\varepsilon$ in small scale, but there is a global advantage to keep $R_\varepsilon^{\leq N}$ .
\end{remark}

\begin{figure}[h]

\centering \includegraphics[scale=0.4]{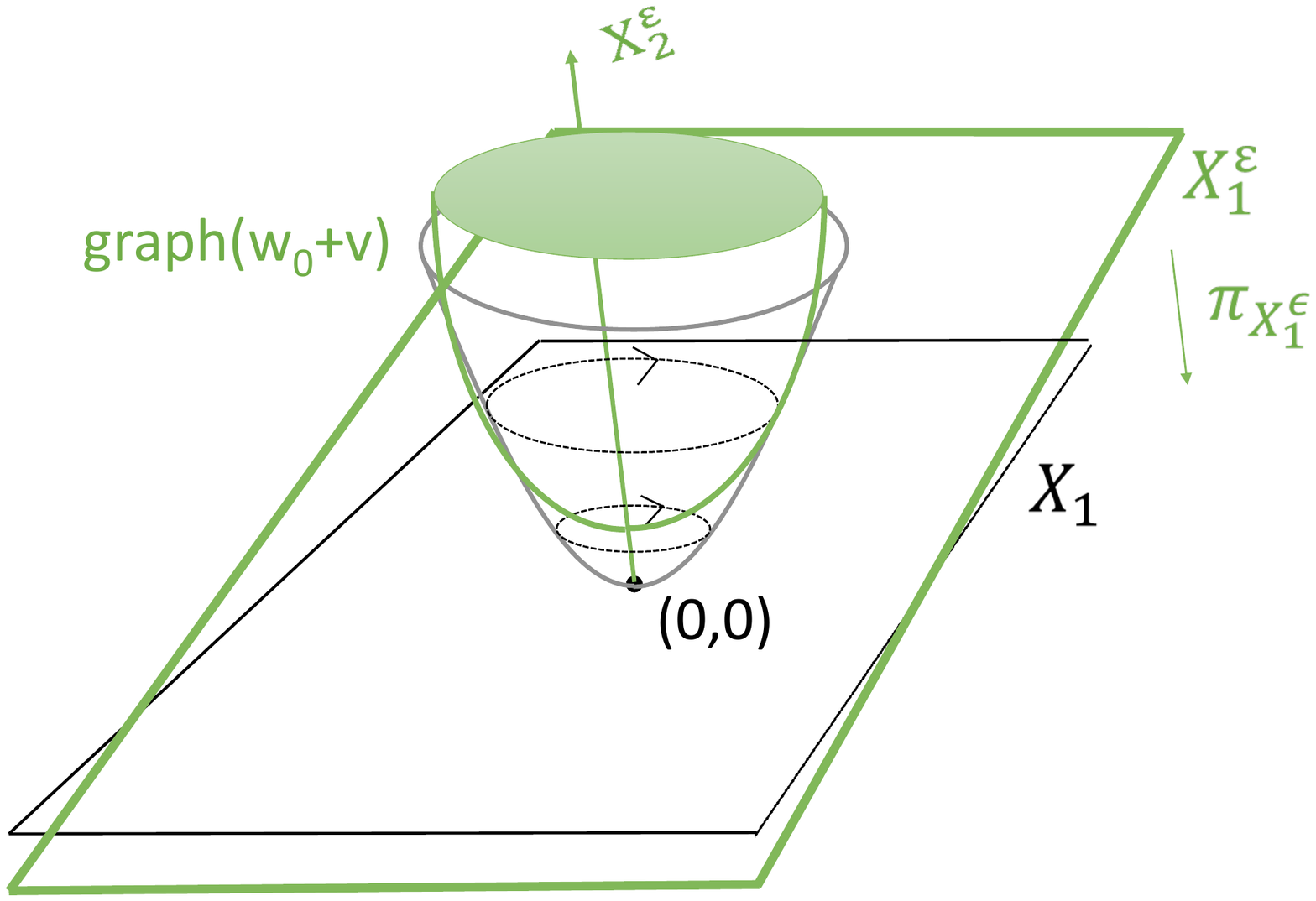}
\caption{The perturbation of the LSM (in green) tangent to the two-dimensional
perturbed subspace $X_{1}^{\eps}$, represented
as the graph of $w_{0}$ plus a small perturbation $x\protect\mapsto v(x,\eps)$.}
\label{ImgPert} 
\end{figure}

\section{Proof of the Main Theorem}
\label{sec:proof}
\subsection{Outline of the Proof}
\label{sec:outlineproof}
In Section~\ref{sec:elementary}, we will derive some immediate consequences of the Assumption ~\ref{Asspert}. After that, in Section~\ref{sec:elementary}, we perform some  preliminary transformations, such as a partial normal form, which will simplify the calculations. In Section~\ref{sec:approximate} we will construct the approximate solutions claimed in part 1 of Theorem~\ref{mainThm}. Again, we emphasize that the main difficulty is that we need to get solutions in a domain of size $\delta> 0$ which is independent of the dissipation parameter $\eps$, so, it has to be a globally defined perturbation expansion.\\
Finally in Section~\ref{sec:fixedpoint} we will reformulate the problem of existence of solutions of \eqref{CH} as a fixed-point problem for an operator defined on the $\A_{\delta,d}^{real}$ spaces introduced before and show that it is a contraction in a small neighborhood of the approximate solution.\\
One subtle point of the contraction argument is that the contraction will be rather weak. Indeed, the contraction  rate will be $ l = 1 - C |\eps| + \O(|\eps|^2)$, for some $C>0$.\\
This weak contraction nevertheless suffices because the approximate solution provided in the first conclusion of Theorem \eqref{mainprop} solves the equation with very high accuracy $O( |\eps|^{N+1})$ (in a domain independent of $\eps$). Then, applying the contraction mapping theorem, we get that the difference between the approximation and the solution is $\mathcal{O}( \frac{1}{1 -l}  |\eps|^{N+1}) = |\eps|^N)$ measured in an appropriate norm for globally defined functions.

\subsection{Domain of attraction, preliminary changes of variables and normalizations}
\label{sec:elementary} 

In this section, we collect some rather elementary results that follow from Assumption~\ref{Asspert}. It seems that these conclusions are the only uses of Assumption~\ref{Asspert} in the proof. Hence, any assumption that leads to them can be used.

\subsubsection{Global domain of attraction} 
\label{sec:global} 

The following lemma will be a consequence of Assumption~\ref{Asspert}. 
It will be important for the transformation of  the invariance equation \eqref{CH} into a fixed point problem.

\begin{lemma} \label{prop:domain} 
Fix any vector field $R_\eps$ that perturbs from a vector field $R_0$ as in the Lyapunov Subcenter Manifold Theorem \ref{exLSC}, satisfying Assumption \ref{Asspert}.\\
Then, for all $\eps \in \C_\th$ (in particular for all  $\eps>0$), we can find a complex domain 
$\BB$, which is mapped into itself by the forward flow of $R_\eps$.
\end{lemma} 

\begin{proof}
Choose again coordinates $(x,y)$ such that the LSM corresponds to the plane $\{y=0\}$. Using the normalization \eqref{suspension} on the LSM, we can assume without loss of generality, that the flow on the LSM is given by a rigid rotation with period $T_0$ and let $r_0:=r_0^{T_0}$ be the corresponding time-$T_0$-map, i.e., $r_0(x)=x$. In particular, we have that $D^2r_0(x)=0$, for all $|x|< \delta$. Denoting the time-$T_0$ map of the perturbed vector field $R_\varepsilon$ as $r_\varepsilon$, it immediately follows that 
\begin{equation}\label{D2r}
D^2r_\varepsilon(x)=\mathcal{O}(\varepsilon),
\end{equation}
 for all $|x|<\delta$.\\
Let $\Lambda_\eps$ be the linear part the time-$T_0$ map of the flow generated by $R_\eps$, i.e.,
\begin{equation}
\Lambda_\eps=Dr_\varepsilon(0).
\end{equation}
By Assumption \ref{Asspert}, we obtain that
\begin{equation}
\sigma( \Lambda_\eps)=e^{\pm\ri\omega_0-\alpha\eps \pm \ri \alpha_I \eps+
\O(\eps^2)}.
\end{equation}
If $\eps \in \C_\th$, $\eps$ is a perturbation of $|\eps|$ in the sense that $\eps - |\eps| = O(\th)|\eps|$ and we can change $\eps$ into $|\eps|$ up to an error which is controlled by $\th$. It follows that
\begin{equation} \label{estLambda}
\begin{split} 
& | \Lambda_\eps| = 1 - (\alpha + \O(\th)) |\eps| +\O(|\eps|^2), \\
& | \Lambda_\eps^{-1}| = 1+ (\alpha + \O(\th))  |\eps| +\O(|\eps|^2), \\
\end{split}
\end{equation}
also because of Assumption~\ref{Asspert}.\\
We can therefore estimate
\begin{equation}
\begin{split}
|Dr_\varepsilon(x)|&\leq|Dr_\varepsilon(0)|+|Dr_\varepsilon(x)-Dr_\varepsilon(0)|\\
&\leq  1 - (\alpha + \O(\th)) |\eps| +\O(|\eps|^2)+|x|\mathcal{O}(\varepsilon)\\
&\leq 1- (\alpha + \O(\th) +\mathcal{O}(\delta))|\eps|+\O(|\eps|^2),
\end{split}
\end{equation}
for $|x|<\delta$, where we have used the (complex) mean-value theorem, as well as the estimates \eqref{estLambda} and \eqref{D2r}. In particular, if $\delta$ and $\theta$ are small enough, the time-$T_0$ map $r_\varepsilon$ is a contraction with contraction factor
\begin{equation}\label{defgamma}
\gamma:=1- (\alpha - \O(\th) -\mathcal{O}(\delta))|\eps|+\O(|\eps|^2)<1.
\end{equation}   
Therefore, the time -$T_0$ map of the flow in a ball of radius $\delta$ has a derivative which agrees with the above up to $ \O(\delta)$ in $B^2_\delta$ and up to  $\O(\delta) + \O(\tau)$ in 
$\BB$, cf. (\ref{deftau}).  Hence, we can get that the spectrum is bounded away from $1$ for all $\delta$ sufficiently small and the conditions of smallness for $\delta$ can be taken independently of $\eps$.
\end{proof}

\begin{remark}
Notice that the above argument essentially uses that the real part of the contraction factor changes with a leading order comparable with $|\eps|$. We anticipate (see Section~\ref{sec:examples}) that, if the contraction was moving more slowly and the nonlinear terms were moving still with $|\eps|$, it would be possible to find periodic orbits in arbitrary small neighborhoods, for small $|\eps|$.  In \cite{de1997invariant}, it was observed that these periodic orbits provide an obstruction to the existence of manifolds with sufficiently high differentiability, in particular, analytic manifolds.\\ 
Of course, even if the real part of the contraction changed more slowly that $|\eps|$, say $\O(|\eps|^2)$, we could recover the result of the global domain of attraction by assuming 
properties of the non-linear terms and the rest of the proof could go through.\\
In many practical problems, the dissipation is a global phenomenon that does not get stopped by the non-linear terms, hence in many practical systems, the conclusions of Proposition~\ref{prop:domain} hold even if Assumption~\ref{Asspert} does not hold (but other global assumptions do). It would be interesting to formulate other general physically meaningful assumptions that account for these phenomena. 
\end{remark} 

\subsubsection{Preliminary changes of variables}
\label{changes}

In this section, we perform some analytic changes of variables that simplify our formulas. Under Assumption \ref{no-resonances}, we have that
	\begin{equation} 
	\label{nonres} | k \lambda_\eps - \mu^\eps_j | \ge  \kappa > 0, 
	\quad  \forall k  \in \integer, \quad |\eps | \le \eps_0,
	\end{equation}
for  some $\eps_0,\kappa>0$ and $1\leq j\leq n-2$.\\
This follows because the imaginary parts for $ k \lambda_\eps - \mu^\eps_j $ differ by a constant for $|k| > d$. We can check that for $k$ small, we cannot generate any new resonances. For large $k$ they cannot be generated either, since then the imaginary parts of $k \lambda_\eps$ and $\mu^\eps_j$ are very different.
\medskip

First, we can adjust our coordinates $(x,y)\in X_1\oplus X_2$ such that the LSM invariant for $F_0$ corresponds to one of the coordinates. That is to say, in this coordinate system, we have that the embedding is just $K_0(x) = (x, 0) $. The plane $\{y=0\}$ then defines an invariant manifold for the vector field 
\begin{equation}
F_0(x, y) = ( R_0(x), \tilde{A}(x,y)),
\end{equation}
 where $(x,y)\mapsto \tilde{A}(x,y)$, the direction transversal to the $R_0$ field that satisfies $\tilde{A}(x,0) = 0$.\\
We can therefore 
arrange that 
\begin{equation}\label{normalization}
DK_\varepsilon(0) = \Pi_{X_1},
\end{equation}
is a fixed isometric embedding from $\real^2$ (or $\complex^2$) into the invariant space (which we arrange to be the first components of the space). The normalization \eqref{normalization} indicates that the embedding $K_\varepsilon$ will be in the affine space $\Pi_{X_1} + \mathcal{A}_{\delta,d}^{real}$.\\
Note that, in particular, in these coordinates, $DK_\eps(0)$ will be independent of $\eps$. In the same vein, we can arrange  that $D R_\eps(0)$ is  the constant map corresponding to the eigenvalues $\lambda^{\pm}_\eps$. In contrast with $DK_\eps(0)$, $DR_\eps(0)$ does depend on $\eps$.
\medskip

We can also make sure that the conserved quantity on the LSM is just given by $|x|^2$. From the results in \cite{kelley1969analytic,SiegelM}, we know that there exists a system of coordinates, such that $R_0$ takes the form
\begin{equation}\label{R0} 
R_0(x)=\left(\begin{matrix}
0 & \Omega(|x|^2)\\ -\Omega(|x|^2) & 0
\end{matrix}\right)x,
\end{equation}
with $\Omega(0)=\omega_0$, for $\omega_0\in\real$, being the linear frequency of the Lyapunov mode.\\
Proceeding as in the theory of normal forms \cite{murdock2003normal,takens1971partially}, for any $N\in\nat$, we can take advantage of the assumed absence of resonances, cf. \eqref{nonres}, and change variables polynomially in $x$ and analytically in $\eps$, such that, separating the variables into $x$, the tangent to the LSM and $y$, the tangent to the complementary space, we have
\begin{equation} \label{normalform}
F_\eps(x,y) = (L + \eps C)(x,y) + A_\varepsilon(x)y + \O(|y|^2|x|^{N+1}),
\end{equation}
for some matrix $A_\varepsilon$. Note that the above normal form can be done uniformly for 
all $\eps$ sufficiently small.
 
\subsection{Approximate solutions to the invariance equation \eqref{CH}}
\label{sec:approximate} 

\subsubsection{Construction of approximate solutions} 
In this section, we construct the sequences $K_j, R_j$ introduced in Theorem~\ref{mainprop}, based on perturbation theory (generalized averaging theory). We again recall that the important feature is that the size of the domain where the approximation is obtained is independent of the size of the dissipative parameter.\\
The perturbation theory we use is valid with uniform bounds in a neighborhood in the 
$x$ variables which is independent of $\eps$. The perturbation theory uses essentially the assumption that there is a LSM consisting of periodic orbits and is basically a mildly sophisticated version of the averaging method on periodic orbits.  Note that this global perturbation theory is very different from the perturbation theory based on normal forms which just matches expansions near the origin of coordinates; these local expansions are very good near the origin, but the region they describe depends on $\eps$.  For our purposes, it is crucial that we can obtain estimates in a region of $x$ which is independent of $\eps$. \\
We remark that the methods used in this section are rather elementary extensions of averaging theory and that they work just as well for finitely differentiable vector fields. 

\medskip 
By assumption, equation \eqref{CH} is satisfied for $\eps=0$ by the Lyapunov Subcenter Theorem, i.e.,
\begin{equation}
F_0(K_0(x))=DK_0(x)R_0(x).
\end{equation}
Formally expanding $F_\eps$, $K_\eps$ and $R_\eps$ in $\eps$,
\begin{equation}\label{approx}
\begin{split}
&F_\eps(x)=\sum_{j=0}^NF_j(x)\eps^j+\mathcal{O}(\eps^{N+1}),\\
&K_\eps(x)=\sum_{j=0}^NK_j(x)\eps^j+\mathcal{O}(\eps^{N+1}),\\
&R_\eps(x)=\sum_{j=0}^NR_j(x)\eps^j+\mathcal{O}(\eps^{N+1}),\\
&F_0(0)=0,\quad K_0(0)=0,\quad R_0(0)=0,
\end{split}
\end{equation}
we see
that equation \eqref{CH} at order $\eps$ becomes
\begin{equation}\label{Oeps}
DF_0(K_0(x))K_1(x)+F_1(K_0(x))=DK_0(x)R_1(x)+DK_1(x)R_0(x).
\end{equation}
The terms $F_0$ and $F_1$ are given by the right-hand side of the differential equation, while $K_0$ and $R_0$ are given by the shape and the dynamics of the LSC, respectively. The unknowns of \eqref{Oeps} are $K_1$, $R_1$ which are, respectively, the first order corrections to the shape of the invariant manifold and to the dynamics on it.\\
More generally, we claim (see the justification below) that expanding to higher order in $\eps$ and matching terms of order $\varepsilon$, we are led to 
\begin{equation}\label{Oepsn}
(DF_0)(K_0(x)) K_n(x)-(DK_n)(x)\, R_0(x)=F_n(K_0(x))+DK_0(x)R_n(x)+S_n(x),
\end{equation}
where $S_n(x)$ is a polynomial expression involving $K_1,...,K_{n-1}$ and $R_1,...,R_{n-1}$, as well as their derivatives.\\
We will study the equations \eqref{Oepsn} recursively. We will show that if  $K_1,...,K_{n-1}$ and $R_1,...,R_{n-1}$ -- and hence $S_n$ -- are known, we can find $(K_n, R_n)$ solving \eqref{Oepsn}.\\
Hence, we will consider \eqref{Oepsn} as an equation for $K_n$, $R_n$ when all the other quantities are known. We anticipate that the solutions $K_n, R_n$  will not be unique, which is consistent with the fact that the equations \eqref{CH} are underdetermined, cf. Remark~\ref{underdetermined}.

\medskip
To prove \eqref{Oepsn}, we first note that 
\begin{equation}
\begin{split}
\left.\frac{\partial^n}{\partial\eps^n}\Big(DK_\eps(x)R_\eps(x)\Big)\right|_{\eps=0}
&=\left.\sum_{j=0}^n\frac{\partial^j DK_\eps}{\partial\eps^j}(x)
\frac{\partial^{n-j}R_\eps}{\partial\eps^{n-j}}(x)\right|_{\eps=0}\\
&=DK_n(x)R_0(x)+DK_0(x)R_n(x)+\sum_{j=1}^{n-1} DK_j(x)R_{n-j}(x).
\end{split}
\end{equation}
Next, we prove inductively that
\begin{equation}
\frac{\partial^n}{\partial\eps^n}F_\eps(K_\eps)=\frac{\partial^n F_\eps}{\partial\eps^n}(K_\eps)+\frac{\partial F_\eps}{\partial x}\frac{\partial^nK_\eps}{\partial\eps^n}+P_n(K_\eps,F_\eps,...),
\end{equation}
where $P_n$ is a polynomial in $F_\eps$, $K_\eps$ and their derivatives up to order $n-1$ in $\eps$. For $n=1$, we obtain \eqref{Oeps}. Assuming the formula for $n$, we obtain
\begin{equation}
\begin{split}
\frac{\partial^{n+1}}{\partial\eps^{n+1}}F_\eps(K_\eps)&=\frac{\partial}{\partial\eps}\left(\frac{\partial^n F_\eps}{\partial\eps^n}(K_\eps)+\frac{\partial F_\eps}{\partial x}\frac{\partial^nK_\eps}{\partial\eps^n}+P_n(K_\eps,F_\eps,...)\right)\\
&=\frac{\partial^{n+1} F_\eps}{\partial\eps^{n+1}}(K_\eps)+\frac{\partial^{n+1} F_\eps}{\partial\eps^n\partial x}(K_\eps)\frac{\partial K_\eps}{\partial\eps}+\frac{\partial F_\eps}{\partial x}\frac{\partial^{n+1}K_\eps}{\partial\eps^{n+1}}\\
&+\frac{\partial^2 F_\eps}{\partial x\partial\eps}\frac{\partial^nK_\eps}{\partial\eps^n}+\frac{\partial^2 F_\eps}{\partial x^2}\frac{\partial^nK_\eps}{\partial\eps^n}\frac{\partial K_\eps}{\partial\eps}+DP_n(K_\eps,F_\eps,...)\cdot\left(\frac{\partial K_\eps}{\partial\eps},\frac{\partial F_\eps}{\partial\eps},...\right),
\end{split}
\end{equation}
where the last bracket contains derivatives in $\eps$ up to order $n$. This proves \eqref{Oepsn}.
\medskip

Now we turn to analyzing \eqref{Oepsn}. Equation \eqref{Oepsn} defines a linear, first-order system of PDEs of the form
\begin{equation}\label{cohomology}
M(x)K_n(x)-DK_n(x)R_0(x) + \eta(x) = 0,
\end{equation}
where the unknowns are $K_n$ and all the other elements, i.e., $M$, $R_0$ and $\eta$, are known with $K_n:\real^2\to\real^N$, $R_0:\real^2\to\real^2$ and $\eta:\real^2\to\real^N$
(again, all of them extend to a complex domain). 
In our case, we have that
\begin{equation}\label{etadefined}
\begin{split}
&\eta(x)= - DK_0(x)R_n(x)-F_n(K_0(x))-S_n(x),\quad M(x)=DF_0(K_0(x)).
\end{split}
\end{equation}
The coefficients in the left-hand side of equation \eqref{cohomology} are the same for all $n$, i.e., $M(x)$ and  $R_0(x)$  do not depend upon $n$.\\
We will develop a theory for general $\eta$ and discover that, to have a solution, $\eta$ has to satisfy some constraints. For our problem, $\eta$ contains the unknown $R_n$. Hence, we will determine $R_n$ so that $\eta$ satisfies the compatibility conditions for the existence of $K_n$ and hence, we can determine $K_n$. Similar procedures (one of the unknowns is determined so that compatibility conditions are met) happen in many perturbative theories in mechanics, cf. \cite{BogoliubovM,  Giacaglia}. They seem to have originated in the perturbative expansions in Celestial Mechanics. We will apply this procedure only a finite number of times (bigger than $d$ in \eqref{defL}). Our only goal is to produce an approximate solution and we do not need (indeed we do not expect) that the series converges.

\subsection{Study of the cohomology equation \eqref{cohomology}}
\label{sec:cohomology} 

In this section, we analyze the cohomology equation \eqref{cohomology}. The main result is to identify the obstructions in $\eta$ for the existence of solutions $K_n$. Later, we will study how to apply this obstructions to find $R_n$ solving \eqref{Oepsn}.\\
Notice that \eqref{cohomology} simply says that if $x(t)$ is a solution of $\dot x = R_0(x)$ then
\begin{equation}\label{periodic} 
D K_n(x(s)) R_0(x(s)) = \frac{d}{ds}  K_n(x(s))   = M(x(s)) K_n(x(s)) + \eta(x(s)) 
\end{equation} 
The solutions of $\dot x = R_0(x)$ are precisely the periodic Lyapunov orbits. Hence,
\eqref{periodic} is a linear equation with periodic coefficients and periodic forcing. If $K_n$ has to be a function of the point $x(s)$ it has to be a periodic function of time of the same period as the orbit $x(s)$.\\ 
Hence, we study the  system:
\begin{equation}\label{char}
\begin{split}
&\frac{dx}{ds}(s)=R_0(x(s)),\\
&\frac{du}{ds}(s)=  \eta(x(s))+ M(x(s))u(s).
\end{split}
\end{equation}
The first equation in \eqref{char} can be written in polar coordinates as
\begin{equation}\label{polar}
\begin{split}
&\frac{d\rho}{ds}(s)=0,\\
&\frac{d\th}{ds}=-\Omega(\rho(s)^2),
\end{split}
\end{equation}
for $x(s)= \Big(\rho(s)\cos(\th(s),\rho(s)\sin{\th(s)}\Big)$ and its solution is given by
\begin{equation}
x(s)=\rho_0\Big(\cos(\th_0-\Omega(\rho_0^2)s),\sin(\th_0-\Omega(\rho_0^2)s)\Big).
\end{equation}
Note that the equations \eqref{char} are just the equations of variation around Lyapunov orbits 
subject to some forcing. See \cite{JorbaV98,Masdemont05,Capinski12} for numerical treatments of \eqref{char} in celestial mechanics. Note also that equations at all orders are equations of the same form.\\
Since the trajectory of $s\mapsto x(s)$ is periodic of period $T(\rho_0)=\frac{2\pi}{\Omega(\rho_0^2)}$ in order for $s\mapsto u(s)$ to define a function of $x(s)$, the solution to the second equation in \eqref{char} must be periodic.\\
More precisely, for a given $\rho_0$, we have to find a condition on the function $s\mapsto\eta(x(s;\rho_0))$, such that the equation
\begin{equation}\label{charu}
\begin{split}
&\frac{du}{ds}(s;\rho_0)= \eta(x(s;\rho_0))+M(x(s;\rho_0))u(s),\\
&u(0;\rho_0)=u_0(\rho_0),
\end{split}
\end{equation}
has a $T$-periodic solution. Since the equation  is a non-homogeneous 
linear equation, the standard variation of parameters formula gives

\begin{equation}\label{perturbation}
u(t)=\Phi(t;t_0)\left[u_0 + \int_{t_0}^t \Phi(s;t_0)^{-1}\eta(s)\, ds\right],
\end{equation}
where $\Phi(t;t_0)$ is the fundamental solution of 
the non-autonomous homogeneous problem
\begin{equation}\label{homogeneous}
\frac{d}{dt} \Phi(t; t_0) = M(x(t)) \Phi(t,t_0); \quad \Phi(t_0; t_0) = \Id
\end{equation}
For typographical reasons, we omit the dependence on $\rho_0$. \\
To have a periodic solution, i.e., $u(t+T)=u(t)$, it suffices to show that $u(t_0)=u(t_0+T)$, where $T$ is the period of the Lyapunov orbit or, explicitly,
\begin{equation}\label{eqy0}
u_0 =\Phi(t_0+T;t_0)\left[u_0 + \int_{t_0}^{t_0+ T} \Phi(s;t_0)^{-1}\eta(s)\, ds\right].
\end{equation}
Rearranging \eqref{eqy0} gives
\begin{equation}\label{eqy0rearranged}
\begin{split} 
\left[ \Phi(t_0+T;t_0) - \Id \right] u_0 &= 
\Phi(t_0+T;t_0)\int_{t_0}^{t_0+ T} \Phi(s;t_0)^{-1}\eta(s)\, ds \\
& = \int_{t_0}^{t_0 + T} \Phi(t_0 +T, s) \eta(s) \, ds.
\end{split} 
\end{equation}
Since $\Phi(t_0 + T, t_0)$ is the linearization of the unperturbed flow under the unperturbed flow, we see that the spectrum of $\Phi(t_0 + T, t_0)$ will contain two eigenvalues $1$ (one of them corresponding to the direction of the flow and another one corresponding to 
the conservation of the energy).\\
To analyze the $n-2$ remaining Lyapunov exponents, we observe that, if we fix a sufficiently small neighborhood in the Lyapunov manifold, the matrix $M(x(s;\rho_0))$ will be a small perturbation of the constant matrix $L$ and that the period $T$ is close to $ 2 \pi/\omega_0$. Hence the spectrum of $\Phi(t_0 + T; t_0)$ will be close to the spectrum of $\exp( \frac{2 \pi}{\omega_0} L)$.\\
Putting the two remarks together, we conclude that in a neighborhood of the origin, the spectrum of $\Phi(t_0 + T; t_0)$  contains two eigenvalues which are exactly $1$ and the remaining $n -2$ are close to $\exp( 2 \pi \ri \frac{\mu_k}{\omega_0})$, which is bounded away from one because of the non-resonace assumptions for Lyapunov orbits.\\
Equation \eqref{eqy0rearranged}, therefore, can be solved if and only if, the right-hand side has no components over the eigenspaces corresponding to the eigenvalues $1$ (identified 
before as the direction of the flow and the gradient of the energy). This is the obstruction in the solution of $K_n$. In the following, we ill show that we can choose the $R_n$'s so that this is soluble.\\
We proceed as in the proof of the Lyapunov theorem and take a surface of section in a coordinate axis in the $X_1$ space. This eliminates the eigenvalue $1$ corresponding to 
the flow. Another way to interpret this is to observe that all the points in the periodic orbit are solutions, so that we always get a one-dimensional family of solutions.\\
On the other hand, for the existence of a solution of \eqref{eqy0rearranged}, there is a true obstruction for $\eta$. We need that the projection to the right-hand side along the direction where the energy vanishes. Note that this is a linear function in $\eta$. We will deal with this obstruction in the next paragraph.

\subsection{Algorithm for the iterative step of 
perturbative expansions} 
\label{sec:iterativestep} 

Using the theory of the cohomology equation as derived in the previous
section, we can device an algorithm to solve to solve recursively the 
equations for $K_n, R_n$. To do
so, we first determine $R_n$ so that $\eta$ in the right-hand side of
\eqref{etadefined} satisfies the constraints needed for the existence
of $K_n$. Then, we determine $K_n$ using the formulas
\eqref{cohomology}. Note that in this selection, the dependence on the periodic orbit 
becomes very important.\\
By choosing the energy $I$ as one coordinate, one coefficient of the matrix $M$ is identically zero, due to energy conservation in the unperturbed system. Also, in our system of coordinates, the matrix $DK_0$ is 
the identity, so that, using \eqref{etadefined}, the condition for the existence of a periodic orbit becomes
\begin{equation}\label{obstructionsolved} 
\int_0^{\frac{2 \pi}{\Omega(\rho_0)}} \Pi_E  R_n ( x(s, \rho_0) )\, ds   = 
-\int_0^{\frac{2 \pi}{\Omega(\rho_0)}}\Pi_E 
(F_n( x(s, \rho_0)  +  S_n(x(s; \rho_0) )\, ds,
\end{equation}

where $\Pi_E$ denotes the projection in the direction of 
the energy with respect to the eigenvalues.\\ 
Due to the underdetermined nature of the the invariance equation \eqref{CH}, we may choose several functions $R_n$. For the sake of simplicity, we choose it to be constant and obtain

\begin{equation} \label{average} 
R_n (\rho_0)  = -\frac{1}{T(\rho_0)}\int_0^{T(\rho_0)}\Pi_E 
(F_n( x(s, \rho_0)  +  S_n(x(s; \rho_0) )\, ds.
\end{equation} 

\begin{remark} 
For $n=1$, \eqref{average} recovers the results of the well known averaging method or the Melnikov theory. 
We can, therefore, interpret $R_n$ for $n>1$ as higher order extensions of Melnikov's method.\\ 
\end{remark}
\begin{remark}
The case of $n = 1$ in the above derivation can also be 
obtained by more familiar averaging arguments or fast/slow 
variables. Since these methods are more familiar in the 
mechanical systems community, we outline them here.\\
We observe that the conserved 
quantity of the unperturbed system  is an slow variable for the
perturbed system, as it evolves with a speed $O(\eps)$. 
Denoting by $x^\eps(t)$ the orbits of the perturbed system, we see that
\[
\begin{split} 
\frac{d}{dt} I( x^\eps(t) )  &= (\nabla I  )(x^\eps(t))  
\cdot F_\eps( x^\eps(t))\\
&= (\nabla I  )(x^\eps(t))  \cdot 
( L x^\eps(t) + N(x^\eps(t) )  + \eps (  C  x^\eps (t)  + G(x^\eps(t))  ) \\
&= \eps (\nabla I  )(x^\eps(t))  (  C  x^\eps (t)  + G(x^\eps(t)) ).
\end{split} 
\]
The change of energy on a cycle is then 
given by 
\[
\begin{split}
\int_0^T \frac{d}{dt} I( x^\eps(t) ) &= 
\eps \int_{0}^T(\nabla I  )(x^\eps(t))  (  C  x^\eps (t)  + G(x^\eps(t)) ) \\
&= \eps \int_{0}^T(\nabla I  )(x^0(t))  (  C  x^0 (t)  + G(x^0(t)) )  
+ O(\eps^2) .
\end{split} 
\]
where we used that, by the smooth dependence on parameters, during the 
finite interval $[0,T]$, we have 
$|x^\eps(t) - x^0(t)| = O(\eps)$.\\
Hence, we approximate, at first order,  the evolution of of the energy over 
a cycle by the average. \\ 
If we seek for the invariant manifold to be given by  selecting 
the normal variables as a function of the energy, we see that since 
this function will be of order $\eps$, the invariant manifold 
will be obtained by selecting the normal variables to be periodic. 
\end{remark}

\subsubsection{Some analytic considerations} 

Now we finish the proof of the first conclusion in Theorem~\ref{mainprop}. 
We examine carefully the formal solutions obtained in the previous section  and 
obtain the desired estimates in the appropriate function space. 
Since the procedure is going to be applied a finite number of 
times, we will not need very detailed estimates.\\
We proceed by induction, assuming that the $K_1,\ldots, K_{n-1}$ and
$R_1,\ldots, R_{n-1}$ are in the appropriate spaces
and we want to conclude the same for the
$K_n$ and $R_n$. \\
First of all, we argue that $K_n$ and $R_n$ are (complex) differentiable 
away from the origin. The differentiability of
the average is clear. The differentiability of $y_0$ --
the initial  condition in the transversal section  -- 
with respect to $\rho$ follows 
from the fact that it is a solution of the implicit equation 
\eqref{eqy0rearranged} whose
coefficients depend differentiability on  the 
radius $\rho_0$. The differentiability with respect to 
the angle is clear for $R_n$ and for $K_n$ it follows because 
it solves a differential equation.\\
This shows that the  function $R_n$ is also differentiable at $\rho_0 =0$. 
For the function $K_n$, we argue similarly. We note that the choice of 
$y_0$ is also differentiable as a function of $\rho_0$ 
(we are inverting a matrix which is clearly differentiable). 
Then, the propagation \eqref{cohomology} is also differentiable along the angle. Again, we use the assumption that the forcing terms vanish to $\O(\rho^{d }) $. This can, indeed, be achieved thanks to the normal form \eqref{normalform} and by noting that, since $F_n$ only has terms of order $d$ and higher, also the composition of $F_n$ with a function hat does not have any constant terms is of order $d$ or higher. By the same token, any algebraic function that involves terms of this form has the desired property, implying that $S_n$ vanishes up to order $\O(\rho^{d }) $.\\
Finally, to obtain the estimates claimed in Theorem \eqref{mainprop}, we 
find that the recursive solution procedure gave functions $K_1, \ldots, K_N$ as well as functions 
$R_1,\ldots R_N$, which  are uniformly differentiable. Also, 
the function
\begin{equation}  \label{residual} 
F_\eps \circ K^{\le N} - DK^{\le N}  R^{\le N}, 
\end{equation}
is differentiable in $\eps$ for fixed $x \in B_\delta^2$.\\
Since the functions $K^{\le N}$ and $R^{\le N}$ have been chosen to match the derivatives with respect to $\eps$ of the invariance equation up to order $N$, the first conclusion of Theorem~\ref{mainprop} follows. 

\begin{remark}
\label{rem:finiteregularityexpansion} 

We note that the results obtained here apply also 
to the case that $F_\eps$ is only finitely differentiable, jointly in 
$\eps$ and $x$.  If the function is $C^\ell$ jointly in $\eps$ and $x$ we see 
that the equations at order  $j$ involve $C^{\ell-j}$ functions
and the algebraic functions of the previously computed solutions. 

The solutions are obtained using only soft arguments such as implicit function 
theorems and hence, the solutions are as smooth as the right hand side. 
Therefore, by induction, we obtain that the $K_j$'s and $R_j$'s are $C^{k -j}$
and that the of the expansion up to order $N$ is $O(|\eps|^{N+1})$ small 
in the sense of $C^{\ell -N -2}$. 
\end{remark}

\subsection{An alternative approach to the theory of the cohomology equation and its analytic estimates using Fourier series}
\label{sec:Fourier} 
In several applications, it is convenient to develop an approach for \eqref{cohomology} based on Fourier series \cite{JorbaV98,Masdemont05}. We recall that a function $\phi(x)$ is analytic in a neighborhood of the origin if and only if it admits an expansion 

\begin{equation}
\phi(x) = \sum_{n_1, n_2 \in \nat} \phi_{n_1,n_2} x_1^{n_1} x_2^{n_2}.
\end{equation}
Using polar coordinates $x_1 = \rho \cos(\th), x_2 = \rho \sin(\th)$ we see that 
\begin{equation}
\phi(x) =  \sum_{n_1, n_2 \in \nat} \phi_{n_1,n_2} \rho^{n_1+ n_2} 
\cos(\th)^{n_1} \sin(\th)^{n_2} 
= :\sum_{n \in \nat} \rho_n  \sum_{k \in \integer, 
|k| \le n} e^{\ri k \th} = :
\sum_{k \in \integer} \phi_k(\rho)  e^{\ri k \th},
\end{equation}
where $\phi_k(\rho)$, $\rho_n$ and $\phi_{n_1,n_2}$ are related through the binomial theorem and finite summation.\\
As it is well known, functions are analytic in a non-trivial domain if and only if the coefficients decrease exponentially. A certain exponential rate in the decrease of the coefficients implies analyticity in a domain and analyticity in a domain implies an exponential rate of decrease of the coefficients. The conditions are not exactly symmetric, but this does not matter for us, since we will only use the procedure a finite number of times.\\
To study \eqref{cohomology}, we can observe that, for each fixed value of $\rho$, the equation \eqref{charu} is a linear periodic equation in the time variable $s$. For sufficiently small $\rho$, the matrix $M$ is a perturbation of a constant coefficient equation. It follows from Floquet theory \cite{Chicone06} that for a fixed $\rho$, we can perform a linear, $T$-periodic change of variables in such a way that the matrix becomes independent of $s$. \footnote{Of course, in the general Floquet theory, we may need to make a $2T$-periodic change of variables, but in our case, for small $\rho$ the differential equation is a perturbation of  the constant equation with the matrix $L$ as a right-hand side, so that the reducibility matrix is periodic and depends analytically on the parameter $\rho$.} Furthermore, the change of variables can be chosen in a way which depends analytically on $\rho$.\\
Hence,  the equation \eqref{cohomology} is equivalent to
\begin{equation}
2 \pi \ri \Omega(\rho) k  \phi_k(\rho)  - A(\rho) \phi_k(\rho) = \eta_k(\rho).
\end{equation}
For $k \ne 0$, the above equation can be solved  because $2\pi \ri \Omega(\rho) k \notin\sigma(A(\rho))$ for all $\rho$ in an small neighborhood. Hence, we just set
\begin{equation}
\phi_k(\rho) = 
(2 \pi \ri \Omega(\rho) k  \phi_k(\rho)  - A(\rho) \phi_k(\rho))^{-1}
\eta_k(\rho).
\end{equation}
For $k=0$, the equation amounts to 
\begin{equation}
A(\rho) \phi_0(\rho) = \eta_0(\rho).
\end{equation}
As indicated, in Section~\ref{changes} we have that there is an eigenvalue zero of $A(\rho)$ corresponding to the change of energy and, by the non-resonance assumption and the perturbation arguments, this is the only zero eigenvalue. Hence we obtain, again, that the obstruction is just that the average of the change of energy of $\eta$ vanishes.\\
There is a constant $C>0$ such that 
\begin{equation}
\|2 \pi \ri \Omega(\rho) k  \phi_k(\rho)  - A(\rho) \phi_k(\rho))^{-1}\| 
\le C.
\end{equation}
Therefore, if $\eta$ satisfies the obstruction and is an analytic function a domain, then the solution of \eqref{cohomology} is analytic in a slightly smaller domain. 

\begin{remark} 
The above Fourier analysis procedure also works for finitely differentiable functions, but the results are weaker than those obtained by the method of integral equations.  We know that if $\phi$  is $C^\ell$, then the Fourier coefficients satisfy $|\phi_n| \le C n^{-\ell}$ for some constant $C>0$. The approximate  converse is that if $|\phi_n| \le C n^{-\ell- \tau}$ for some $\tau > 1$ then $\phi \in C^\ell$.\\
Hence, by working with Fourier coefficients to analyze \eqref{cohomology}, we obtain estimates with $ 1+ \tau$ derivatives less. For the purposes of this section, this is not a fatal loss since we only need to apply it a finite number of times. On the other hand, it would be a very useless estimate for a fixed point argument. One can, however, avoid this shortcoming by working in Sobolev spaces. A comparison between Fourier methods and integral formulas  
for closely related problems appears in \cite{HuguetL13}.
\end{remark}

\subsection{The Fixed-Point Argument}
\label{sec:fixedpoint} 
Throughout this section, we will assume that the unperturbed vector field has been normalized according to \eqref{normalization} and that we have chosen coordinates $(x,y)\in\mathbb{R}^2\times\mathbb{R}^{n-2}$, such that the LSM corresponds to the invariant plane $\{y=0\}$.\\
We start by transforming equation \eqref{CH} into an equivalent form, suitable for a fixed point argument. Let $\phi^t_\eps$, either defined as a function $\phi_\eps^t:\real^n\to\real^n$ or $\phi_\eps^t:\complex^n\to\complex^n$, be the flow map associated to the vector field $F_\eps$, i.e.,
\begin{equation}
\frac{d}{dt} \phi^t_\eps = F_\eps \circ \phi^t_\eps, \qquad \phi^0_\eps =
{\rm Id}.
\end{equation}
Analogously, let $r^t_\eps$, either defined as a function $r_{\eps}^t:\real^n\to\real^n$ or $r_{\eps}^t:\complex^n\to\complex^n$, be the flow associated to the vector field $R_\eps$, i.e.,
\begin{equation}
\frac{d}{dt} r^t_{\eps} = R_\eps \circ r^t_\eps, \qquad r^0_\eps =
\text{Id}.
\end{equation}

To simplify notation, we denote the corresponding time-$T_0$ maps, cf. \eqref{normalization}, as $\phi_\eps = \phi^{T_0}_\eps$ and $r_\eps = r^{T_0}_\eps$. \\
The invariance equation \eqref{CH} is then equivalent to 
\begin{equation}\label{invflow}
\phi^t_\eps(K_\eps(x))=K_\eps(r^t_\eps(x)), 
\end{equation}
for all $t\geq 0$.\\
Consider equation \eqref{invflow} only for $t = T_0$ and rewrite it as 
\begin{equation} \label{CH2} 
K_\eps(x)=\phi_\eps^{-1} \circ K_\eps(r_\eps(x)).
\end{equation} 
Later, we will show that the solution of \eqref{CH2} also solves \eqref{invflow} and, hence \eqref{CH}. To establish existence of solution of \eqref{CH} we will show that the operator defined by the right-hand side of \eqref{CH2} is a contraction in a ball around the approximate solution produced in Part 1) of Theorem~\ref{mainThm}.\\
As we already have found an approximate solution $K_\eps^{\leq N}(x)$ in both $x$ and $\eps$ \eqref{approx}, we can reformulate \eqref{invflow} as
\begin{equation}
K_\eps^{\leq N}(x)+K^{>N}_\eps(x)=\phi_\eps^{-1}\Big(K_\eps^{\leq N}(r_\eps(x))+K_\eps^{>N}(r_\eps(x))\Big).
\end{equation}
That is to say, $K^{>N}_\eps$ should be a fixed point of the $\eps$-dependent functional
\begin{equation}\label{Taudefined} 
\mathcal{T}_{\eps}(\hK)(x)=
\phi_\eps^{-1}\circ\Big(K_\eps^{\leq N}(r_\eps(x))+ \hK(r_\eps(x))\Big)  -
K_\eps^{\leq N}(x).
\end{equation}
We emphasize that the approximate solution $K_\varepsilon^{\leq N}$ has been obtained on all $x$ in a neighborhood which is independent of $\varepsilon$ since we have just integrate along periodic orbits.\\
We will first show that for all $\eps \in \C_\th$,  $\mathcal{T}_\eps$ defined in \eqref{Taudefined}, maps a ball in $\A^{real}_{\delta, d}$ to itself and is a contraction. After that, we will study the dependence of the fixed point on $\eps$ and show that the fixed point is analytic in $\eps$ for $\eps \in \C_\th$ and that, near zero, it has an asymptotic expansion.  We also recall, see Remark~\ref{rem:negativeeps}, that, by reversing the time, we can also study the case $\eps \in -\C_\th$. \\
In particular, we will obtain that, if we consider $\eps \in \real$ -- the physically more 
interesting case -- we have that the fixed point as a function of $\eps$ is real analytic for $\eps\in \real \setminus \{0\} $ and $C^\infty$ at $\eps = 0$. 

\smallskip
First, we claim that $\mathcal{T}_\eps:B_{\mathcal{A}_{\delta,d}^{real}}^\sigma\subseteq\mathcal{A}_{\delta,d}^{real}\to\mathcal{A}_{\delta,d}^{real}$, for 
\begin{equation}
B_{\mathcal{A}_{\delta,d}^{real}}^\sigma=\{K\in\mathcal{A}_{\delta,d}^{real}:\|K\|_{\mathcal{A}_{\delta,d}}<\sigma\},
\end{equation}
is well-defined.\\
Indeed, by the second statement in Assumption \ref{Asspert}, $\BB$ is mapped into itself by $r_\eps$ for all $\eps \in \C_\th$. Therefore $K(r_\eps(x))$ is well-defined for $|x|<\delta$.
We also note that $D( \phi^{-t}_\eps \circ K^{<N} \circ  r_\eps^t)(0)  = \Id$ 
and that if $K$ vanishes to high order, we get that $ \mathcal{T}_\eps(K)$ also satisfies the normalization of the derivatives \eqref{normalization}.\\
To see that $\mathcal{T}_\eps(K(x))=\O(|x|^{d})$, it suffices to employ the coordinate system presented in Section \ref{changes} and the fact that $K_\eps^{\leq N}$ solves the invariance equation up to order $d$ in $x$. Clearly, $\mathcal{T}_\eps(K)|_{\real^2}\subseteq\real^N$.\\
Since $K_\eps^{\leq N}$ is an approximate solution to \eqref{invflow} up to oder $|\eps|^N$, it follow that 
\begin{equation}
\|\mathcal{T}_\eps(0)\|_{\mathcal{A}_{\delta,d}}=\O(|\eps|^N).
\end{equation}
We will assume that the size of the ball $B_{\mathcal{A}_{\delta,d}^{real}}^\sigma$ in function space is $\sigma=\varepsilon^M$, i.e.,
\begin{equation}
\hat{K}(x)=\mathcal{O}(\varepsilon^M),	
\end{equation}
for all $|x|<\delta$ and for some $M>1$, indicating that the correction to the formal expansions will be small. The exact value of $M$ will be determined in the course of the proof.\\
To estimate the contraction rate of the functional $\mathcal{T}_\eps$, we calculate
\begin{equation}
\begin{split}
\|\mathcal{T}_\eps(\hK_1)&-\mathcal{T}_\eps(\hK_2)\|_{\mathcal{A}_{\delta,d}} =\sup_{z\in \BB}|z|^{-d}\left|\phi_\eps^{-1}\Big(K_\eps^{\leq N}(r_\eps(z))+\hK_1(r_\eps(z))\Big)-\phi_\eps^{-1}\Big(K_\eps^{\leq N}(r_\eps(z))+\hK_2(r_\eps(z))\Big)\right|\\
&\leq \Lip(\phi_\eps^{-1})\sup_{z\in \BB}|z|^{-d}|\hK_1(r_\eps(z))-\hK_2(r_\eps(z))|\\
&\leq \Big[1+(\beta+\mathcal{O}(\theta))|\varepsilon|+\mathcal{O}(|\varepsilon|^2)\Big]\sup_{z\in \BB}|z|^{-d}|\hK_1(r_\eps(z))-\hK_2(r_\eps(z))|.
\end{split}
\end{equation}
 Here, we have used Lemma \ref{uniform} together with \eqref{estDphi-1} and the fact that, in our coordinate system,
\begin{equation}
K^{\leq N}_\varepsilon(z)+\hat{K}(z)=K_0(z)+\mathcal{O}(\varepsilon)=(z,0)+\mathcal{O}(\varepsilon).
\end{equation}
Also, we have used again that $\varepsilon=(1+\mathcal({\theta}))|\varepsilon|$.
To proceed, we estimate the contraction rate of the perturbed reduced dynamics as
\begin{equation}
\begin{split}
\|\mathcal{T}_\eps(\hK_1)&-\mathcal{T}_\eps(\hK_2)\|_{\mathcal{A}_{\delta,d}}\\
&\leq\Big[1+(\beta+\mathcal{O}(\theta))|\varepsilon|+\mathcal{O}(|\varepsilon|^2)\Big]\sup_{z\in \BB}|z|^{-d}|r_\eps(z)|^{d}|r_\eps(z)|^{-d}|\hK_1(r_\eps(z))-\hK_2(r_\eps(z))|\\
&\leq \Big[1+(\beta+\mathcal{O}(\theta))|\varepsilon|+\mathcal{O}(|\varepsilon|^2)\Big]\gamma^d\sup_{z \in \BB}|z|^{-d}|z|^{d}\sup_{z \in \BB}|r_\eps(z)|^{-d}|\hK_1(r_\eps(z))-\hK_2(r_\eps(z))|\\
&\leq \Big[1+(\beta+\mathcal{O}(\theta))|\varepsilon|+\mathcal{O}(|\varepsilon|^2)\Big]\gamma^{d}\|\hK_1-\hK_2\|_{\mathcal{A}_{\delta,d}},
\end{split}
\end{equation}
where $\gamma$  is the contraction factor introduced in 
\eqref{defgamma}.\\
Expanding the contraction factor to leading orders in 
$|\eps|$, we obtain 
\begin{equation}
\|\mathcal{T}_\eps(\hK_1)-\mathcal{T}_\eps(\hK_2)\|_{\mathcal{A}_{\delta,d}} \leq \Big( 1 +(\beta+\mathcal{O}(\theta)-d[\alpha - \O(\th) -\mathcal{O}(\delta)])|\varepsilon|+\mathcal{O}(|\varepsilon|^2 )\Big) \|\hK_1-\hK_2\|_{\mathcal{A}_{\delta,d}},
\end{equation}
which, by \eqref{defL}, implies that $\mathcal{T}_\eps$ is a contraction for $\eps$, $\th$ and $\delta$ small enough and $d$ big enough such that
\begin{equation}
\beta+\mathcal{O}(\theta)-d[\alpha - \O(\th) -\mathcal{O}(\delta)]<0.
\end{equation}

\begin{remark}\label{rem:uniformcontraction}
It will be important for future applications to observe that, after we fix the $\th$, $\delta$ and $\beta$ the contraction is uniform for all values of $\eps \in \C_\th$ that satisfy $0 <  a_- <  | \eps| <  a_+$, for some constants $a_-$ and $a_+$. Of course, these uniform rate of contraction becomes close to $1$ as $a_-$ converges to zero and we cannot obtain a uniform contraction for all $\eps \in \C_\th$.
\end{remark}

To see that $\mathcal{T}_\varepsilon\Big(B_{\mathcal{A}_{\delta,d}^{real}}^\sigma\Big)\subseteq B_{\mathcal{A}_{\delta,d}^{real}}^\sigma$, for some $\sigma=\sigma(\eps)>0$, we observe that, for $\eps$ sufficiently small,
\begin{equation}\label{sigmaest}
\begin{split}
\|\mathcal{T}_\eps(\hK)\|_{\mathcal{A}_{\delta,d}}&=\|\mathcal{T}_\eps(\hK)-\mathcal{T}_\eps(0)+\mathcal{T}_\eps(0)\|_{\mathcal{A}_{\delta,d}}\\
&\leq \|\mathcal{T}_\eps(\hK)-\mathcal{T}_\eps(0)\|_{\mathcal{A}_{\delta,d}}+\|\mathcal{T}_\eps(0)\|_{\mathcal{A}_{\delta,d}}\\
&\leq  \Big( 1 +\Big[\beta+\mathcal{O}(\theta)-d[\alpha - \O(\th) -\mathcal{O}(\delta)]\Big]|\varepsilon|+\mathcal{O}(|\varepsilon|^2 )\Big)\|\hK\|_{\mathcal{A}_{\delta,d}}+\mathcal{O}(\eps^N)\\
&\leq \sigma\Big( 1 +\Big[\beta+\mathcal{O}(\theta)-d[\alpha - \O(\th) -\mathcal{O}(\delta)]\Big]|\varepsilon|+\mathcal{O}(|\varepsilon|^2 )\Big)+C_2|\eps|^N\\
&\leq \sigma,
\end{split}
\end{equation}
for some $C_2>0$. To ensure the last inequality in \eqref{sigmaest}, it suffices to take $\eps$ sufficiently small and 
\begin{equation}
\sigma\geq\frac{C_2 |\eps|^{N-1}}{d[\alpha - \O(\th) -\mathcal{O}(\delta)]-\beta-\mathcal{O}(\theta)-\mathcal{O}(|\varepsilon|^2)}>0.
\end{equation}
To finish the proof of Theorem~\ref{mainThm}, we have to show that the solutions of \eqref{CH2} also solve \eqref{invflow}. Note that the main difference between \eqref{CH2} and \eqref{invflow} is that \eqref{invflow} is only the evolution of $r_\varepsilon$ for one time, while \eqref{CH2} involves the evolution for all times. To achieve this, we use an argument coming from \cite{Cab2003}. If $K_\eps$ is a solution of \eqref{CH2}, for any $s \in\real$ sufficiently small, we have that
\begin{equation}
\begin{split} 
\phi^s_\eps \circ  K_\eps \circ r^{-s}_\eps & =
\phi^s_\eps \circ  \phi_\eps^{-1}  \circ  K_\eps  \circ r_\eps \circ r^{-s}_\eps\\ 
& = \phi_\eps^{-1} \circ   \phi^s_\eps  \circ  K_\eps    \circ r^{-s}_\eps \circ r_\eps \\  
& = \mathcal{T}_\eps(\phi^s \circ K_\eps \circ r^{-s}_\eps),
\end{split} 
\end{equation}
by the flow property of $\phi_\eps^s$. Hence, we obtain that $\phi^s_\eps \circ K_\eps \circ r^s_\eps$ is also a fixed point of $\mathcal{T}_\eps$ and it also satisfies the normalization conditions specified in our main theorem. For $s$ small enough, $\phi^s_\eps \circ K_\eps \circ r^s_\eps$, will be in the domain of uniqueness of the fixed point theorem. Therefore, we obtain that there exists an interval of $s$ such that 
\begin{equation}
\phi^s_\eps \circ K_\eps \circ r^{-s}_\eps = K_\eps,
\end{equation}
which implies \eqref{invflow}.\\
To prove the analyticity in $\eps$ for $\eps \in \C_\th$, we recall that the contraction properties are the same for all the $\eps \in \C_\th$ such that $0 < a_- < |\eps| < a_+$, for some $a_-,a_+>0$, and that around any point, we can find a set $\U$ -- without of loss of generality inside the previous one -- so that for all values in this set, there is a ball of radius $\sigma > 0$ in $\A^{real}_{d,\delta}$ that gets mapped into itself, see Remark~\ref{rem:uniformcontraction}.\\ 
We also observe that if $\tilde K_\eps$ is analytic in $\eps$ for $\eps$ in such a domain, 
we also obtain that $\mathcal{T}_\eps( \tilde K_\eps)$ is also analytic in $\eps$ -- it suffices to apply the chain rule for derivatives. Putting these two remarks together, we obtain that $\mathcal{T}^n_\eps(0)$ is a sequence of uniformly converging analytic functions and hence their limit is analytic in the domain. Therefore, the fixed-point depends analytically on $\eps$.

\begin{remark} The contraction argument also works for the case that the $F_\eps$ is $C^\ell$ if $\ell > d$. The Lipschitz constants of the operators $\mathcal{T}_\eps$ acting on $C^r$-spaces vanishing to order $d$ are estimated in \cite[Proposition 3.2]{de1997invariant} or \cite{Cab2003}. Hence we obtain that the invariant manifolds produced in Theorem~\ref{mainprop} are locally unique under the condition that the manifolds are invariant and $C^r$ for  $r>d$. Note also that, remembering Remark~\ref{rem:finiteregularityexpansion}, we obtain, rather straightforwardly, an analogue for finite regularity of the existence results claimed in Theorem~\ref{mainprop}.\\
The results of Theorem~\ref{mainprop} on regularity with respect to $\eps$ or with respect to parameters seem to be also true, but they seem to require substantial work (which we will not undertake here).
\end{remark}

\begin{remark}
As a note for experts (which most readers may want to postpone) we note that the operator $\mathcal{T}_\eps$ is differentiable in $\hat K$ in $C^\ell$ spaces (note that this is 
not the case with the operators used in the graph transform approach to NHIHM \cite{Fenichel73}). To obtain differentiability with respect to parameters, the main problem is that, in spaces of finite differentiability, the operator $\mathcal{T}_\eps$ is not differentiable with respect to $\eps$., since the formal derivative with respect to $\eps$ would include a term
$D \phi^{-1}_\eps\circ ( K_\eps^{\leq N} + \hK) \circ r^1_\eps D ( K_\eps^{\leq N} + \hK)\circ r^1_\eps D_\eps r^1_\eps $. Indeed, this derivative involves $D\hK$, so that the derivative in a certain $C^\ell$ space would involve a term that can only be controlled in $C^{\ell +1}$.  This is a well known problem for operators involving composition in the left \cite{LlaveO99}.
As a consequence, the standard implicit function theorem does not apply and one needs to develop more sophisticated methods.\\
For the problem of non-resonant manifolds, a very detailed study of the differentiability with respect to parameters, developing specialized implicit function theorems, appears in \cite{Cab2003pam}. The results of \cite{Cab2003pam} show that the differentiable manifold will be continuously differentiable -- in a rather subtle sense -- for $\eps \in (0, \eps_0)$. Since it satisfies the asymptotic expansions, Theorem~\ref{mainprop} shows also it is differentiable at $\eps = 0$. The question of continuity of the derivative with respect to $\eps$ at $\eps = 0$ is 
significantly more subtle and, as far as we know, does not follow from the literature above. 
\end{remark}

\section{Some mathematical examples}
\label{sec:examples} 
In this section, we present some examples that show that some assumptions of Theorem \eqref{mainThm} are necessary to guarantee the existence of an invariant manifold with the specified properties and having a size independent of the dissipation parameter. \\
In particular, we show that if the eigenvalues of the linear part depend upon $\eps^p$, for some $p>1$, existence of a sufficiently differentiable invariant manifold is no longer guaranteed. Indeed, the $\eps$-dependence of the eigenvalue reflects the singular nature of
the problem. Notice that these examples show that \eqref{defL} is not enough to guarantee the existence of an SMM of size independent of the dissipation. One also seems to need some global information on the shape of the dissipation.\\
We will first describe in Example~\ref{firstexample} a very special  system.  After we understand its basic properties, we will show how to modify the system slightly so that there are obstructions for the existence of SMM of size $1$. 

\begin{example}\label{firstexample} 
Consider the following system in polar coordinates  $\rho,\th$ and $u, \psi$.
\begin{equation}\label{example1}
\begin{split}
&\rho' =  -\eps^2\rho   + \eps \rho^2,\\ 
&\th' = \omega,\\
&u' =  a  \eps^2 u, \\
&\psi' =  \gamma, 
\end{split}
\end{equation}
for 
$(\rho,\th,u, \psi)\in\real^+\times \circle^1\times\real^+\times \circle^1$
and $a \in \real_+$ a number which we will adjust to 
get obstructions to the regularity.\\
Equation \eqref{example1} corresponds to a polynomial vector field in Cartesian coordinates, being a mere rotation for $\eps=0$. The linear part of the system \eqref{example1}  changes only by $\eps^2$,  thus violating assumption \ref{Asspert}. 
If $\gamma$ is not an integer multiple of $\omega$, we can verify the hypothesis of LSM. Indeed, in this very simple example, the LSM is just the plain described by $\rho$ and $\th$. For $\varepsilon=0$, the energy is $\rho^2+u^2$, which is non-degenerate. The equation for $\psi$ plays no role and we can avoid it in much of the analysis. For simplicity, we have chosen in the first of \eqref{example1} only a perturbation which is quadratic in $\rho$, which leads to some coincidences which are not really relevant for the analysis and could be removed by adding more complicated nonlinearities. The key point is that the nonlinear term is affected by $\eps$ and takes the trajectories out of the origin while the linear term is affected by $\eps^2$ and 
takes the trajectories in, even if very slowly. For  $\eps> 0$, the origin is a stable, hyperbolic fixed-point. The local  stable manifold of the hyperbolic fixed-point is precisely the plane $u,\th$. Moreover, there exists an unstable limit cycle at $\rho=\eps$, $u = 0$, which we denote as $P$. The Lyapunov exponent of $P$ in the $\rho$-direction is given by $\eps^2$, while in the $u$-direction, it is given by $a\eps^2$ (the variable $\psi$ makes that the eigenvalues  of the return map in the periodic orbit are not just the Lyapunov exponents but that they acquire a phase.)\\
Now that we have understood the geometry of \eqref{example1}, we can construct perturbations that do not have any analytic SMMs near the LSM.  Note that the perturbations we construct now are perturbations of the function $F_\eps(x)$. That is, we will change the function of two variables $\eps$ and $x$, keeping, of course, the normalizations. 
As a concrete way to understand the perturbations of the family we can imagine adding an extra parameter $\nu$ and that we add terms containing $\nu\eps$ to the family.

\begin{remark} 
Note that the analysis so far already shows that the proofs of existence of SMM based on the graph transform will have a problem even to get started, cf. Example 5. By the unperturbed center flow, there  is no domain of size $1$ in $u$ that gets mapped into itself. In analytic regularity, we cannot cut off and extend, so that the graph transform proofs for analytic manifolds have problems even being formulated. Even for $C^2$-regularity, we cannot find domains that are mapped to themselves.\\
Of course, the failure of a method of proof does not automatically imply the failure of the conclusions (e.g. the folding of the manifold gives problems to graph methods but poses 
no problem to parameterization methods), but it certainly gives a hint of the problems. Later, we will see that one can exclude the existence of SMMs of high enough regularity. 
\end{remark} 

We start by observing that for any perturbation,  the origin will be hyperbolic and that the only invariant manifold near the origin close to the stable manifold of the unperturbed system is the unstable manifold. Hence, any SMM has to coincide with the stable manifold near the origin. We also note that any possible SSM of the perturbation has to go through the periodic orbit $P$. This is because the periodic orbit $P$ has a basin of repulsion of size $\eps$. Any invariant object that does not include $P$ has to be out of the basin of repulsion.\\ 
The key observation is that, since $P$ has eigenvalues that do not resonate with those of the complement, there is a non-resonant invariant manifold near $P$. This manifold is unique under the assumption that it is sufficiently differentiable, cf. \cite{de1997invariant, Cab2003,CABRE2005444}. We note that the space spanned by the non-resonant eigenvectors is close to the tangent of the LSM. This non-resonant manifold is persistent under small perturbations. We observe that the non-resonant manifold near $P$ is the only candidate for an invariant manifold close to the LSM for $\eps = 0$. Hence, it is the only candidate for being the SMM of size $1$ near the LSM which is sufficiently differentiable. Therefore, we have two conditions that the SSM or size of order $1$ has to satisfy: It has to agree with the stable manifold near the origin and it has to be the non-resonant manifold near $P$.\\
The coincidence of these two manifolds is an infinite codimension phenomenon in the space of 
maps. If they indeed coincided, a generic small perturbation will break this coincidence, as the perturbation theory for these two manifolds lead to very different Melnikov functions, cf. \cite{de1997invariant} for more details.  Some similar examples for parabolic manifolds with numerical computations appear in \cite{baldoma2007parameterization}. Notice also that the above examples show that finite differentiable manifolds of size of order one cannot exist for general perturbations that do not satisfy our assumptions. Note also that the obstructions cannot be easily seen by studying just the jet of the manifold at the origin. One can also consider slightly more complicated perturbations for the $u$-equation in \eqref{example1} so that one can get  several attracting/hyperbolic periodic orbits appearing with the perturbation.
\end{example}

\begin{remark}
The key to make the example \eqref{example1} work is that the non-linear terms push out of the origin with a coefficient $\eps$ while the linear terms push in with a coefficient $\eps^2$.  The possibility of the example disappears if one adds the extra assumption that all the terms push in. It would be interesting to investigate if indeed under some more strict global assumption one can recover the result of existence of SSMs. This result would be physically interesting since many models of physical interest seem to satisfy the extra assumption.
\end{remark}

\begin{example}\label{non-uniform}
Consider a system of the following form,
\begin{equation}\label{hyper}
\begin{split}
&\dot{x}=Jx-2|y|^2\text{tr}\Gamma(|x|^2)\frac{x}{|x|^2},\\
&\dot{y}=\Gamma(|x|^2)y,
\end{split}
\end{equation}
where $x\in\mathbb{R}^2$, $y\in\mathbb{R}^{2k}, k>1$, $J$ being the standard symplectic matrix and 
\begin{equation}
\Gamma(|x|^2)=\diag(\Gamma_1(|x|^2),...,\Gamma_k(|x|^2)),
\end{equation}
for $2\times 2$ matrices $\Gamma_1,...,\Gamma_k$ such that $\text{tr}(\Gamma(0))=0$ as well as $D_x\text{tr}(\Gamma(0))=0$, implying that $\frac{x\text{tr}\Gamma(|x|^2)}{|x|^2}$ is differentiable at $x=0$. By construction, the scalar function $I(x,y)=|x|^2+|y|^2$ is a conserved quantity for system \eqref{hyper}. Clearly, at the same time, system \eqref{hyper} need not be Hamiltonian. If the matrices $\Gamma_1,...,\Gamma_k$ are now chosen in such a way that $\text{tr}(\Gamma(x))\neq 0$ for $x\neq 0$, then system \eqref{hyper} can admit hyperbolic directions, showing that the assumption a symplectic map in Lemma \ref{uniform} cannot be omitted. The LSM is thus, in this example, a normally hyperbolic invariant manifold with weakening hyerbolicity as $x\to 0$. Note also that the hyperbolicity properties can change in $|x|\approx\varepsilon$ in this model.
\end{example}

\begin{remark} 
The examples above exclude not only the existence of analytic manifolds of size of order one but they also exclude the existence of $C^\ell$ manifolds of size of order $1$ for large enough $\ell$. One question that deserves more exploration is whether one can get invariant manifolds with very low regularity even in the cases where we cannot get very differentiable manifolds. 
\end{remark}

\begin{remark} \label{rem:easy}
The examples above allow us to compare the results obtained in this paper with other possible alternatives and will serve the experts to understand some of our choices.\
To study general perturbations of LSMs (even those not satisfying Assumption~\ref{Asspert}),  we could apply the theory \cite{de1997invariant, Cab2003,CABRE2005444} and, for every value of $\eps \in \complex$ leading to $\Re(\lambda_\eps) < 0$, produce an analytic invariant 
manifold. Since the coefficients can be computed recursively, it is easy to see that these manifolds converge to the LSM in the sense that each of the derivatives at zero converge. It seems that one could obtain similar results using \cite{Moulton20}. Going through the proofs in \cite{de1997invariant, Cab2003,CABRE2005444}, one can see that the fact that the contraction at the origin is weak for small $\eps$ results in the straightforward reading of the results applying only to functions defined in smaller domains as $\eps$ goes to zero.\\
The examples above show that this decrease of the domain is not an artifact of the proof.  Indeed, to obtain manifolds defined in a domain of size independent of $\eps$ we need not only to take advantage of the conserve quantity for the unperturbed case, but also of some sort of global information on the perturbations such as Lemma~\ref{prop:domain}. This seems to require that the Floquet multipliers remain elliptic in a neighborhood.\\
Proofs based only on local information of power series near the origin will only produce 
results in small domains unless one take advantage of subtle cancellations that will depend on other hypothesis. The physical meaning of the manifolds produced in this paper is clear since they serve as long paths that guide the systems to equilibria \cite{GorbanK05, Smooke91}. We do not know what could be the physical meaning of the manifolds obtained by power matching (in the cases when they do not coincide with those in the theorem here). 
\end{remark}

\section{Applications and examples}
\label{sec:applications}
Next, we compare Theorem \ref{mainThm} to the following theorem on the existence of
two-dimensional spectral submanifolds for autonomous dynamical systems
in \cite{Haller2016}:\\
\begin{theorem}\label{SSM} Consider a two-dimensional spectral
subspace $X_{1}$ of the linearization $L+\eps C$ of
equation (\ref{main}) associated
with the eigenvalues $\lambda_{\eps},\overline{\lambda}_{\eps}$.
Assume that $\text{Re}(\sigma(L+\eps C))<0$ and assume
that the non-resonance condition 
\begin{equation}
k\lambda_{\eps}+l\overline{\lambda}_{\eps}\neq\mu,\label{nonresSSM}
\end{equation}
for all $k,l\in\integer$ 
and all eigenvalues 
$\mu\in\sigma(L+\eps\tilde{L})\setminus\{\lambda_{\eps},\overline{\lambda}_{\eps}\}$,
holds. Then, there exists a two-dimensional, analytic, invariant manifold
$W_{\eps}$, tangent to the spectral subspace $X_{1}$ around
the trivial solution $X=0$.
\end{theorem}
The proof of the above
theorem is based on a theory of invariant spectral manifolds in Banach
spaces derived in \cite{Cab2003pam}, \cite{Cab2003} and \cite{CABRE2005444}. Since the non-resonance condition in eq. \eqref{nonresSSM}
is weaker than the non-resonance condition \eqref{nonres}, a spectral
submanifold guaranteed by Theorem \ref{SSM} may not converge
to the LSM of Theorem \ref{exLSC} for $\eps\to 0$, as the
following example shows.
\begin{example}(Not all SSMs are perturbed LSMs.)
Consider the dynamical system 
\begin{equation}
\begin{split} & \dot{\xi}=\left(\begin{matrix}-\eps\lambda & 1\\
-1 & -\eps\lambda
\end{matrix}\right)\xi,\\
 & \dot{\eta}=\left(\begin{matrix}-\eps\mu & \alpha\\
-\alpha & -\eps\mu
\end{matrix}\right)\eta+F(\xi),
\end{split}
\label{ex1}
\end{equation}
for the functions $t\mapsto\xi(t),\eta(t)\in\real^{2}$, some
nonlinearity $F(\xi)=\mathcal{O}(|\xi|^{2})$ and the parameters
$\lambda,\mu\in\real^{+},\eps>0,$
as well as $\alpha\in\integer$. The spectrum of the linearization
is given by $\{-\eps\lambda\pm\ri\alpha,-\eps\mu\pm\ri\alpha\}$.
We assume that $\lambda$ and $\mu$ are independent over the integers,
i.e., that there is no $n\in\integer$ such that $\lambda=n\mu$.
We conclude from Theorem \ref{SSM} the existence of an analytic,
two-dimensional, invariant manifold tangent to the $\xi$-plane for any
$\eps>0$. For small enough $\xi$, we can write $\eta$ as
a function of $\xi$, i.e., 
$\eta(t)=H(\xi(t))$, for $H(\xi)=\sum_{|n|=2}^{\infty}H_{n}\xi^{n}$,
$\xi=(\xi_{1},\xi_{2}),n=(n_{1},n_{2})$. Substituting the expression
for $\eta$ as a function of $\xi$ into equation (\ref{ex1}) and
expanding $F(\xi)=\sum_{|n|=2}^{\infty}F_{n}\xi^{n}$, we obtain at
order $|\xi|^{2}$: 
\begin{equation}
\left(\begin{matrix}B_{\eps} & 0 & I_{2}\\
0 & B_{\eps} & -I_{2}\\
-2I_{2} & 2I_{2} & B_{\eps}
\end{matrix}\right)\left(\begin{array}{c}
H_{20}\\
H_{02}\\
H_{11}
\end{array}\right)=\left(\begin{array}{c}
F_{20}\\
F_{02}\\
F_{11}
\end{array}\right),\label{inverse}
\end{equation}
where we $B_{\eps}=\left(\begin{matrix}\eps(2\lambda-\mu) & \alpha\\
-\alpha & \eps(2\lambda-\mu)
\end{matrix}\right)$ and $I_{2}$ denotes the $(2\times2)$-identity matrix. For the choice
$\alpha=2$, the inverse of the matrix in the right-hand side of (\ref{inverse})
will contain terms of order $\eps^{-1}$. Therefore, assuming
that $(F_{20},F_{02},F_{11})\neq(0,0,0)$, we find that the $\mathcal{O}(|\xi|^{2})$-terms
in the expansion of $H$ will blow up as $\eps\to0$. Indeed,
we use the inversion formula for block-matrices, 
\begin{equation}
\left(\begin{matrix}A & B\\
C & D
\end{matrix}\right)^{-1}=\left(\begin{matrix}A^{-1}+A^{-1}B\Delta^{-1}CA^{-1} & -A^{-1}B\Delta^{-1}\\
-\Delta^{-1}CA^{-1} & \Delta^{-1}
\end{matrix}\right),\label{blockinv}
\end{equation}
with matrices $A,B,C,D$ of arbitrary dimensions, $A$ invertible and
$\Delta=D-CA^{-1}B$, to show that the inverse of (\ref{inverse})
contains an entry of order $\eps^{-1}$. Focusing on the lower-right
corner in (\ref{blockinv}), we obtain 
\begin{small}

\begin{equation}
\begin{split}\Delta & =\left(B_{\eps}-(-2I_{2},2I_{2})\left(\begin{matrix}B_{\eps}^{-1} & 0\\
0 & B_{\eps}^{-1}
\end{matrix}\right)\left(\begin{array}{c}
I_{2}\\
-I_{2}
\end{array}\right)\right)\\
 & =(B_{\eps}+4B_{\eps}^{-1})\\
 & =\left(\begin{matrix}\eps(2\lambda-\mu)+4\frac{\eps(2\lambda-\mu)}{\eps^{2}(2\lambda-\mu)^{2}+\alpha^{2}} & \alpha-4\frac{\alpha}{\eps^{2}(2\lambda-\mu)^{2}+\alpha^{2}}\\[0.2cm]
-\alpha+4\frac{\alpha}{\eps^{2}(2\lambda-\mu)^{2}+\alpha^{2}} & \eps(2\lambda-\mu)+4\frac{\eps(2\lambda-\mu)}{\eps^{2}(2\lambda-\mu)^{2}+\alpha^{2}}
\end{matrix}\right)\\
 & =(\eps^{2}(2\lambda-\mu)^{2}+\alpha^{2})^{2}\left(\begin{matrix}\eps(2\lambda-\mu)\Big(\eps^{2}(2\lambda-\mu)^{2}+\alpha^{2}+4\Big) & \eps^{2}(2\lambda-\mu)^{2}\alpha+(\alpha^{3}-4\alpha)\\
-\eps^{2}(2\lambda-\mu)^{2}\alpha+(4\alpha-\alpha^{3}) & \eps(2\lambda-\mu)\Big(\eps^{2}(2\lambda-\mu)^{2}+\alpha^{2}+4\Big)
\end{matrix}\right).
\end{split}
\label{Deltainv}
\end{equation} 
\end{small}
For $\alpha=2$, the $(\alpha^{3}-4\alpha)$-contribution vanishes
and $\Delta$ scales like $\eps$ in its entries, which implies
that $\Delta^{-1}$ scales like $\eps^{-1}$ in its entries.
This shows that the SSMs does not converge to an  LSM in this example
(even in the sense of convergence of Taylor coefficients). 
\end{example} 

\begin{remark}
Note that in this example, the $\eps = 0$ case is resonant and does not verify the hypothesis of the Lyapunov center theorem. Indeed, it is easy to show matching powers of a possible expansion that, indeed, we obtain some equations that have no solution and, hence, that in this case there is no LSM.\\
Therefore the SSM for small dissipation is not generated by the LSM, and the SMM are created by the dissipation and have no chance of surviving in the limit of zero dissipation. As we mentioned in the introduction, as soon as the eigenvalues move into non-resonance, the theory of \cite{Cab2003} guarantees the existence of SSMs whose Taylor coefficients converge to those of the LSM. Hence, the situation exemplified in the example -- no LSM -- is the only one when one could get failure of convergence of the Taylor coefficients.\\
One question that would be interesting to study is what is the domain of convergence of the SSM is in this example. Since the Taylor coefficients blow up, Cauchy estimates indicate that the (complex) domains when the parameterization has size $1$ have to go to zero but we do not know the rate and we do not know what is the nature of the singularities.\\
Extensions of Lyapunov theorem to resonant eigenvalues have been worked out in \cite{weinstein73, moser1976periodic,Duistermaat1972, SchmidtS73} using variational methods and averaging methods. It would also be interesting to apply the results in this case. 
\end{remark}

\subsection{Dissipative perturbations of Hamiltonian systems}

In the following, we discuss the existence of an
SSM, smoothly perturbed from an LSM, for nearly conservative
systems relevant in application to mechanics.\\
Consider the dynamical system 
\begin{equation}
M\ddot{q}+\eps C\dot{q}+Kq+\nabla V(q)=0,\label{mech1}
\end{equation}
for $t\mapsto q(t)\in\real^{n}$, $n\in\nat$, where $M$ is a positive definite
mass matrix, $K$ is a positive definite stiffness matrix, $C$ is a positive definite
damping matrix and $V:\real^{n}\to\real$ a twice continuously-differentiable
function such that $V(0)=0$ and $\nabla V(0)=0$. Introducing the
generalized momentum $p:=M\dot{q}$, we can rewrite system \eqref{mech1}
as a first-order, dissipatively perturbed Hamiltonian system
of the form
\begin{equation}
\left(\begin{array}{c}
\dot{q}\\
\dot{p}
\end{array}\right)=\left(\begin{matrix}0 & M^{-1}\\
-K & -\eps CM^{-1}
\end{matrix}\right)\left(\begin{array}{c}
q\\
p
\end{array}\right)-\left(\begin{array}{c}
0\\
\nabla V(q)
\end{array}\right).\label{mech2}
\end{equation}
For $\eps=0$, equation \eqref{mech1} admits the Hamiltonian
\begin{equation}
H(q,p)=\frac{1}{2}p\cdot M^{-1}p+\frac{1}{2}q\cdot Kq+V(q),\label{eq:hamiltonian}
\end{equation}
such that $\frac{d}{dt}H((q(t),p(t)))=0$ for any solution $t\mapsto(q(t),p(t))$
to (\ref{mech2}) with $\eps=0$. The spectrum of the linearization
of (\ref{mech2}) for $\eps=0$ is given by $\pm\ri\sqrt{\sigma(M^{-1}K)}=:\{\pm\ri\lambda_{k}\}_{1\leq k\leq n}$,
which is purely imaginary thanks to the positive-definiteness of $M$
and $K$. To apply Theorem \ref{exLSC}, we assume that the first eigenvalue of $M^{-1}K$ is non-resonant
with the other eigenvalues, i.e.,
\begin{equation}
\frac{\lambda_{k}}{\lambda_{1}}\notin\integer,
\end{equation}
for $2\leq k\leq n$. We also have to assume that $\nabla^{2}V(Y,Y)>0$
for all $Y\in\eig(\pm\ri\lambda_{1})$. With these conditions, Assumption \ref{AssLSC} is satisfied.\\
 Assumption \ref{Asspert} is satisfied for any
positive definite damping matrix $C$ for which the linear part in (\ref{mech2}) has no repeated eigenvalues. Indeed, the Hamiltonian acts
as a Lyapunov function for system \eqref{mech1}, as $H(0,0)=0$ and 
\begin{equation}
\dot{H}(q,p)=\nabla H\cdot(\dot{q},\dot{p})^{T}=-\eps M^{-1}p\cdot CM^{-1}p<0,
\end{equation}
by the positive-definiteness of $C$. Therefore, the origin is asymptotically
stable and, for $\delta'$ and $\eps$ small enough.

\begin{example}(Nonlinear Elastic Pendulum with Air Drag) Based on Duistermaat
\cite{Duistermaat1972}, consider a two-dimensional elastic pendulum
with mass $m>0$ under the influence of gravity $g>0$. At the initial
position, the mass is assumed to be located at the origin $(0,0)$
and the length of the spring is given by $l_{0}$. We assume a linear
spring constant $k>0$ and a cubic spring constant $K>0$, so that
the potential energy at the position $(q_{1},q_{2})$ is given by
\begin{equation}
U(q)=U(q_{1},q_{2})=mgq_{2}+k[q_{1}^{2}+(l_{0}-q_{2})^{2}]+K[q_{1}^{2}+(l_{0}-q_{2})^{2}]^{2}+U_{0},\label{pot}
\end{equation}
for $U_{0}=-(kl_{0}^{2}+Kl_{0}^{4})$ (cf. Fig. \ref{ImgElastic}).
\begin{figure}[h]
\centering \includegraphics[scale=0.4]{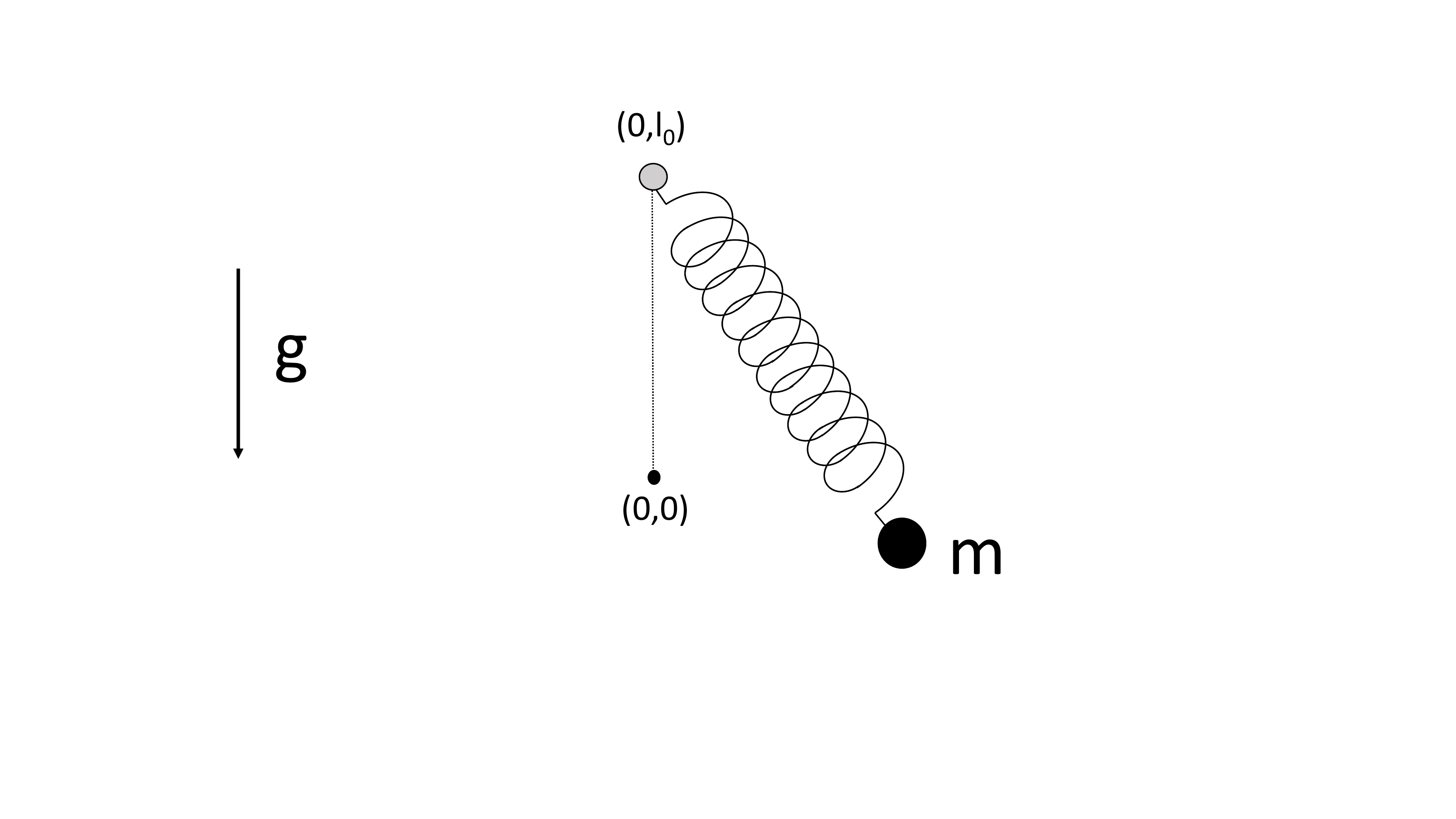}
\caption{Elastic pendulum under the influence of gravity. The initial position
is assumed to be at $(0,0)$, while the initial length is assumed
to be $l_{0}$.}
\label{ImgElastic} 
\end{figure}
The Hamiltonian \eqref{eq:hamiltonian} takes the specific form
\begin{equation}
H(q,p)=\frac{1}{2m}|p|^{2}+U(q),\label{Ham}
\end{equation}
and $H(0,0)=0$ and $\nabla H(0,0)=0$ hold when we choose our parameters
such that $2kl_{0}+4Kl_{0}^{3}=mg$. Assuming a constant, small air
drag $\eps$ (which we consider as the linearization of the general, quadratic air resistance), the equations of motions take the form 
\begin{equation}
\begin{split}\left(\begin{array}{c}
\dot{q}_{1}\\
\dot{q}_{2}\\
\dot{p}_{1}\\
\dot{p}_{2}
\end{array}\right)= & \left(\begin{matrix}0 & 0 & \frac{1}{m} & 0\\
0 & 0 & 0 & \frac{1}{m}\\
-(2k+4Kl_{0}^{2}) & 0 & -\frac{\eps}{m} & 0\\
0 & -(2k+12Kl_{0}^{2}) & 0 & -\frac{\eps}{m}
\end{matrix}\right)\left(\begin{array}{c}
q_{1}\\
q_{2}\\
p_{1}\\
p_{2}
\end{array}\right)\\
 & \quad-\left(\begin{array}{c}
0\\
0\\
4Kq_{1}(q_{1}^{2}-2l_{0}q_{2}+q_{2}^{2})\\
4Kq_{2}(q_{2}^{2}-3l_{0}q_{2}+q_{1}^{2})-4Kl_{0}q_{1}^{2}
\end{array}\right).
\end{split}
\label{pend}
\end{equation}
For $\eps$ sufficiently small, the eigenvalues of the linearization
of (\ref{pend}) are given by 
\begin{equation}
\left\{ \frac{-\eps\pm\ri\sqrt{2k+4Kl_{0}^{2}-\eps^{2}}}{2m},\frac{-\eps\pm\ri\sqrt{2k+12Kl_{0}^{2}-\eps^{2}}}{2m}\right\} ,
\end{equation}
with the $(q_{1},p_{1})$-plane and the $(q_{2},p_{2})$-plane as
corresponding eigenspaces for any $\eps\geq0$.\\
 Assumption \ref{AssLSC}/(3) is therefore satisfied for both pairs of complex conjugate
eigenvalues and their corresponding eigenspaces for $\eps=0$.
To ensure the existence of LSMs for these eigenspaces, we also require
that 
\begin{equation}
\frac{k+2l_{0}^{2}K}{k+6l_{0}^{2}K}\notin\integer,
\end{equation}
which holds for a generic choice of $k$ and $K$, assuming that $l_{0}$
is chosen according to the normalization of the Hamiltonian. We will
now construct a second-order polynomial approximation to the invariant
manifold tangent to the $(q_{1},p_{1})$-plane and its corresponding
polynomial dynamics, which are guaranteed to exists by Theorem \ref{mainThm}.\\
To this end, we expand
$w:U\subset\real^{2}\times[0,\eps_{0}]\to\real^{2}$,
for some $\eps_{0}>0$, up to order two
\begin{equation}
w(\eps,q_{1},p_{1})=w_{20}(\eps)q_{1}^{2}+w_{11}(\eps)q_{1}p_{1}+w_{02}(\eps)p_{1}^{2}+\mathcal{O}((|q|+|p|)^{3}),\label{ansatz1}
\end{equation}
for $\eps\mapsto
w_{20}(\eps),w_{11}(\eps),w_{02}(\eps)\in\real^{2}$,
and assume that
$(q_{2}(t),p_{2}(t))=w(\eps,q_{1}(t),q_{2}(t))$.  Substituting
the ansatz (\ref{ansatz1}) into equation \eqref{pend} and solving for
powers in $q_{1}$ and $p_{1}$, we obtain
\begin{equation}
\begin{split} & w_{20}(\eps)=\kappa_{\eps}\left(\begin{array}{c}
-2Kl_{0}(-\eps^{2}(k+6Kl_{0}^{2})-6k^{2}m+8kKl_{0}^{2}m+8K^{2}l_{0}^{4}m)\\
-16\eps Kl_{0}(k+2Kl_{0}^{2})^{2}m
\end{array}\right),\\[0.3cm]
 & w_{11}(\eps)=\kappa_{\eps}\left(\begin{array}{c}
16\eps Kl_{0}(k+2Kl_{0}^{2})\\
-4Kl_{0}(\eps^{2}(k-2Kl_{0}^{2})+2(3k^{2}+20kKl_{0}^{2}+12K^{2}l_{0}^{4})m)
\end{array}\right),\\[0.3cm]
 & w_{02}(\eps)=\kappa_{\eps}\left(\begin{array}{c}
8Kl_{0}(3k+2Kl_{0}^{2})\\
8\eps Kl_{0}(-k+2Kl_{0}^{2})
\end{array}\right),
\end{split}
\label{wij}
\end{equation}
where 
\begin{equation}
\kappa_{\eps}=\frac{1}{(k+6Kl_{0}^{2})(-\eps^{2}(k-2Kl_{0}^{2})+2(3k+2Kl_{0}^{2})^{2}m)}.
\end{equation}
The approximate dynamics on the SSM (i.e., perturbed LSM) are given
by 
\begin{equation}
\left(\begin{array}{c}
\dot{x}\\
\dot{y}
\end{array}\right)=\left(\begin{array}{c}
\frac{2}{m}y\\
-2(k+2Kl_{0}^{2})x-\frac{\eps y}{m}+4K\kappa_{\eps}^{2}x\left(x^{2}+l_{0}^{2}(\alpha_{\eps}+\beta_{\eps}xy+\gamma_{\eps}x^{2}+\delta_{\eps}y^{2})^{2}\right)
\end{array}\right),\label{red1}
\end{equation}
where 
\begin{equation}
\begin{split} & \alpha_{\eps}:=-(k+6Kl_{0}^{2})(-\eps^{2}(k-2Kl_{0}^{2})+2(3k+2Kl_{0}^{2})^{2}m),\\
 & \beta_{\eps}:=16\eps K(k+2Kl_{0}^{2}),\\
 & \gamma_{\eps}:=2K(\eps^{2}(k+6Kl_{0}^{2})+2(3k^{2}-4kKl_{0}^{2}-4K^{2}l_{0}^{4})m),\\
 & \delta_{\eps}:=8K(3k+2Kl_{0}^{2}).
\end{split}
\end{equation}
Typical phase portraits for the unperturbed and for the perturbed
system (\ref{red1}) are depicted in Figures \ref{Red1Fig} and \ref{Red1epsFig}.

\begin{figure}[h]
	\begin{subfigure}{.5\textwidth}
		\centering
		\includegraphics[width=.9\linewidth]{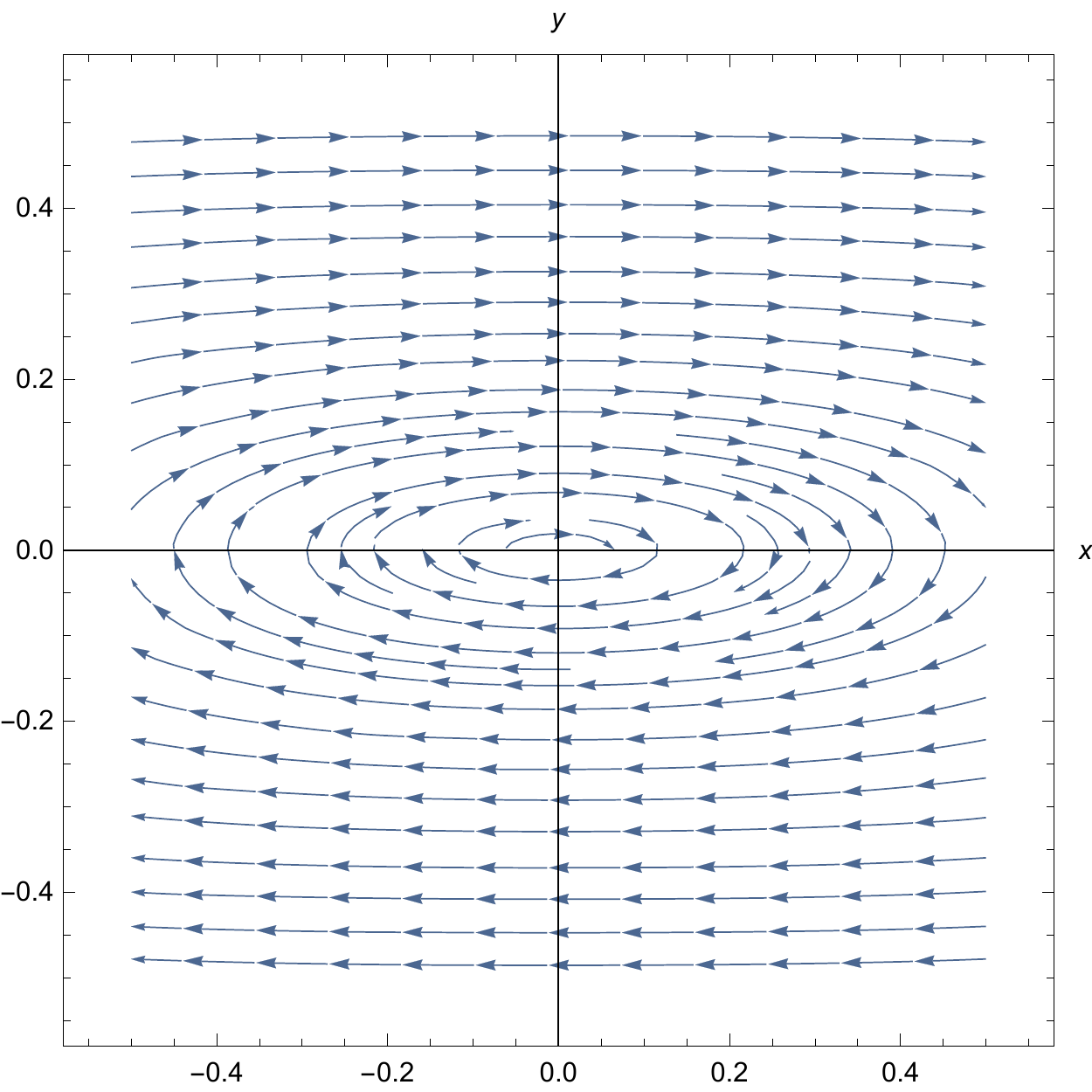}
		\caption{The unperturbed dynamics on the LSM for\\ $l_0=1$, $m=0.1$, $k=0.1414$ and $K=0.4$. }
		\label{Red1Fig}
	\end{subfigure}%
	\begin{subfigure}{.5\textwidth}
		\centering
		\includegraphics[width=.9\linewidth]{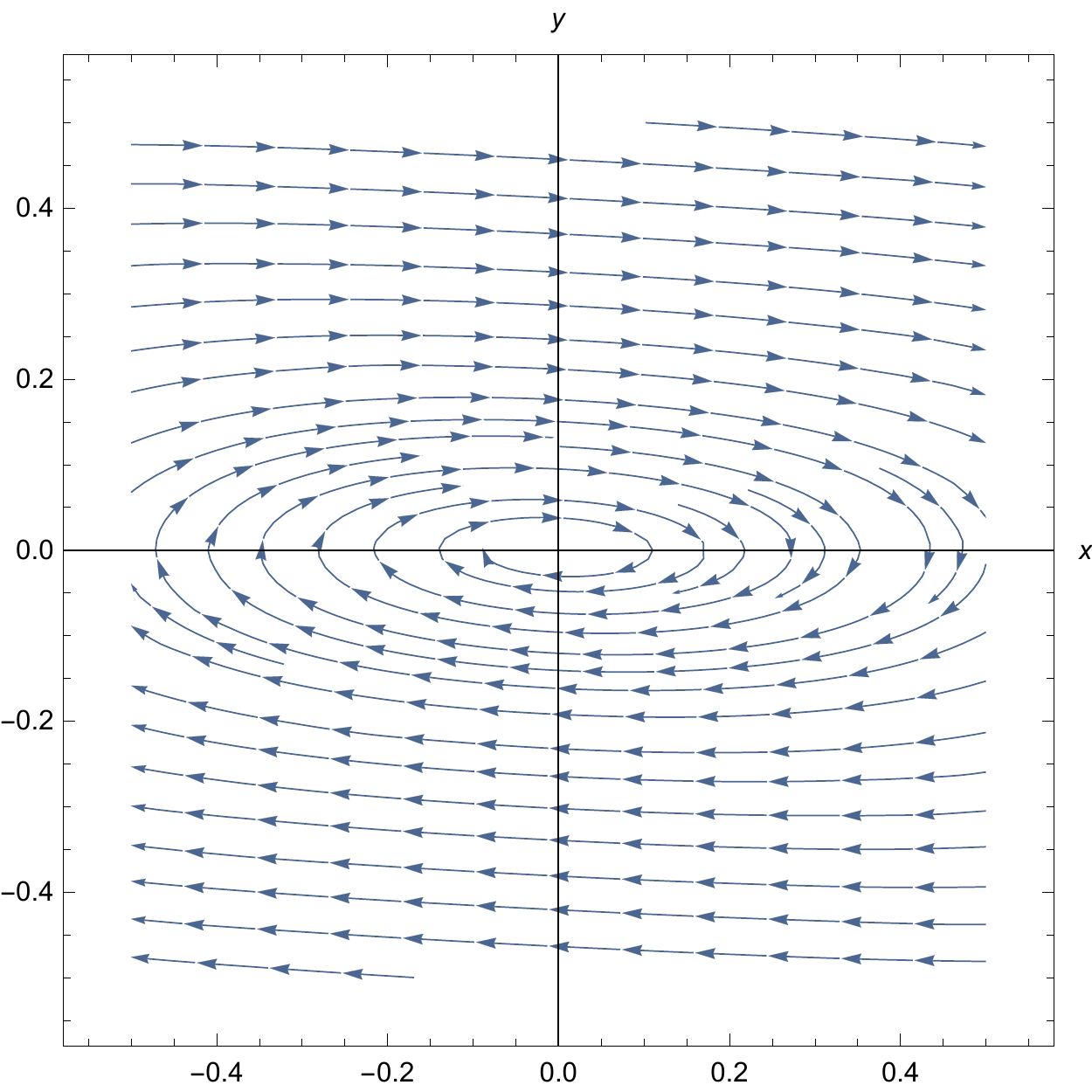}
		\caption{The perturbed dynamics on the SSM for $l_0=1$,\\ $m=0.1$, $k=0.1414$, $K=0.4$ and $\eps=0.1$.}
		\label{Red1epsFig}
	\end{subfigure}
\end{figure}

\end{example}

\subsection{Conclusion and Further Perspectives}
We have proved that, adding dissipation (satisfying some mild non-degeneracy conditions) to a Hamiltonian system, an analytic, two-dimensional Lyapunov Subcenter Manifold (LSM) perturbs to an analytic spectral submanifold (SSM) in a neighborhood of size or order one.\\
As a consequence, the corresponding reduced dynamics on the SSM are $\eps$-close to the dynamics of the LSM  in a neighborhood of size one. We have also illustrated our results on several examples.\\
It would be useful to extend the present results to time-periodic or quasi-periodic perturbations and thereby relate experimentally observed backbone curves under sinusoidal excitation to the backbone curve of the conservative, unforced limit of the system. We refer to \cite{breunung2018explicit} for theoretical and numerical discussion of backbone curve calculations based on the SSM technique for damped systems. Of course, periodic perturbations lead to new phenomena such as resonances that lead to different behaviors and will require new formulations of results.\\
In view of Example~\ref{firstexample} it may also be interesting to study the possibility of existence of $C^\ell$ manifolds of size $\O(1)$ with $\ell < d$. Indeed, in many cases, it has been observed that the manifolds that guide the convergence to equilibrium are of low regularity \cite{MaasP92, Smooke91, WarnatzMD96, GorbanK05}.\\
We also think it would be interesting to consider results in PDE adding dissipation to the results in \cite{Bambusi00} or in systems with symmetry \cite{BuonoLM05}, keeping in mind that the dissipation may also break some symmetries.\\
It seems that the analyticity domains in the dissipation established here can be improved to parabolic domains $\{\eps \in \complex \, \| |\Im(\eps)|  \le  B \Re(\eps)^2\}$. It would be interesting to characterize the optimal analyticity domains and in particular, show that they do not contain a circle centered at the origin.\\
We also refer to upcoming results by Szalai \cite{Szalai2018conservative}, in which a perturbation theory based on different techniques has been announced.

\section*{Acknowledgments} 
We thank George Haller for formulating the problem and suggesting it to us as well as for many conversations and continued interest and encouragement. We also thank Robert Szalai for conversations and for pointing out several mistakes in a previous version.\\
We would also like to thank the anonymous reviewers for several useful comments and suggestions.
\bibliographystyle{abbrv}
\bibliography{DynamicalSystems.bib}

\end{document}